\documentclass{nm}
\renewcommand\thisnumber{x}
\renewcommand\thisyear {200x}

\renewcommand\thisvolume{xx}

\usepackage{charter}
\usepackage[charter]{mathdesign}
\usepackage{multirow}

\begin{document}


\markboth{Yan Gui, Rui Du, Cheng Wang}{3rd-order IMEX-RK Method for LL Equation with Arbitrary Dmping}
\title{A Third-order Implicit-Explicit Runge-Kutta Method for Landau-Lifshitz Equation with Arbitrary Damping Parameters}


\author[Yan Gui, Rui Du and Cheng Wang]{Yan Gui\affil{1}, Rui Du\affil{2}\comma\corrauth and Cheng Wang\affil{3}}
\address{\affilnum{1}\ School of Mathematical Sciences, Soochow University, Suzhou 215006, China\\
\affilnum{2}\ School of Mathematical Sciences, Soochow University, Suzhou 215006, China\\
\affilnum{3}\ Mathematics Department, University of Massachusetts, North Dartmouth, MA 02747, USA}

\emails{{\tt 20204007008@stu.suda.edu.cn} (Y. Gui), {\tt durui@suda.edu.cn} (R. Du), {\tt cwang1@umassd.edu} (C. Wang)}
%

\begin{abstract}
A third-order accurate implicit-explicit Runge-Kutta time marching numerical scheme is proposed and implemented for the Landau-Lifshitz-Gilbert equation, which models magnetization dynamics in ferromagnetic materials, with arbitrary damping parameters. This method has three remarkable advantages:~(1) only a linear system with constant coefficients needs to be solved at each Runge-Kutta stage, which greatly reduces the time cost and improves the efficiency; (2) the optimal rate convergence analysis does not impose any restriction on the magnitude of damping parameter, which is consistent with the third-order accuracy in time for 1-D and 3-D numerical examples; (3) its unconditional stability with respect to the damping parameter has been verified by a detailed numerical study. In comparison with many existing methods, the proposed method indicates a better performance on accuracy and efficiency, and thus provides a better option for micromagnetics simulations.
\end{abstract}



\maketitle


\section{Introduction}\label{S:Introduction}
The Landau-Lifshitz (LL) equation has been widely used to describe the evolution of magnetic order (magnetization) in continuum ferromagnetic materials \cite{Landau1935On, gilbert1955lagrangian}, which is a vectorial and non-local nonlinear system with non-convex constraint in a point-wise sense and possible degeneracy. A crucial issue in the LL equation is to design efficient and high-order numerical schemes, and considerable progresses have been made in the past few decades; see \cite{An2022, kruzik2006recent, CJreview2007, cimrak2007survey, EAJAM-14-601} for reviews and references therein. Explicit algorithms (e.g. \cite{alouges2006convergence,bartels2008numerical}) and semi-implicit schemes (e.g. \cite{weinan2001numerical, wang2001gauss, cimrak2005error, gao2014optimal, boscarino2016high,an2016optimal,An2021, NMTMA-16-182}) are very popular since they avoid a complicated nonlinear solver while preserving the numerical stability, in comparison with the fully implicit ones (e.g. \cite{bartels2006convergence, fuwa2012finite}). 

One typical semi-implicit method is based on the backward differentiation formula (BDF) temporal discretization, combined with one-sided extrapolation for nonlinear terms \cite{xie2020second, chen2021convergence, Lubich2021}. In \cite{chen2021convergence}, the second-order BDF approximation is applied to obtain an intermediate magnetization, and the right-hand-side nonlinear terms are treated in a semi-implicit style with a second-order extrapolation applied to the explicit coefficients. A projection step is further used to preserve the unit length of magnetization at each time step, which poses a non-convex constraint. Such a numerical algorithm, called semi-implicit projection method (SIPM), leads to a linear system of equations with variable coefficients and non-symmetric structure. As a result, no fast solver is available for this numerical system. Meanwhile, an unconditionally unique solvability of the semi-implicit scheme with large damping (SIPM with large damping) has been proved in~\cite{cai2022second}. The improvement is based on an implicit treatment of the constant-coefficient diffusion term, combined with a fully explicit extrapolation approximation of the nonlinear terms, including the gyromagnetic term and the nonlinear part of the harmonic mapping flow. A direct advantage could be observed in the fact that, the resulting numerical scheme only requires a standard Poisson solver at each time step, which greatly improves the computational efficiency. However, an unconditionally stability is only available for large damping parameter $\alpha>1$, while most magnetic material models correspond to a parameter $\alpha\ll 1$. In addition, higher-order BDF methods could be applied, while only the first-order and second-order BDF algorithms are unconditionally stable. As analyzed in \cite{Lubich2021}, for the BDF schemes of orders 3 to 5, combined with finite element spatial discretization, the numerical stability requires the damping parameter to be above a positive threshold: $\alpha > \alpha_k$ with $\alpha_k=0.0913, 0.4041, 4.4348$ for order $k=3,4,5$ respectively. Therefore, it would be highly desirable to design an efficient and higher accurate scheme with no requirement on the damping parameter.

For time-dependent nonlinear partial differential equations in general, implicit-explicit (IMEX) schemes have been extensively used \cite{boscarino2016high}. For the LL equation, the second-order IMEX has been studied in \cite{xie2020second}. Two linear systems, with variable coefficients and non-symmetric structure, need to be solved. Hence IMEX2 can hardly compete with BDF2 in terms of accuracy and efficiency. In a recent  work~\cite{wang2020local}, the authors introduce an artificial linear diffusion term and treat it implicitly, while all the remaining terms are treated explicitly. Afterwards, the second-order and the third-order implicit-explicit Runge-Kutta (IMEX-RK2, IMEX-RK3) methods, in which the popular coefficients are derived by the work~\cite{ascher1997implicit}, were proposed for the LL equation in a recent work~\cite{Gui24a}.  Moreover, extensive numerical results have demonstrated that the IMEX-RK2 method has a better performance over the BDF2 approach, in terms of accuracy and efficiency. These IMEX-RK methods worked well for arbitrary damping, and this is a very significant fact in scientific computing, since the damping parameter may be small in most magnetic materials \cite{Brown1963micromagnetics}. However, the corresponding theoretical analysis becomes a very difficult issue, because of the complicated structure of the RK coefficients. 

In other words, higher-order RK numerical schemes could be appropriately constructed, while, the theoretical analysis of any specific IMEX-RK3 scheme, including the linearized stability estimate and optimal rate convergence analysis, is expected to be much more challenging, due to the complicated coefficient stencil in the Runge-Kutta stages. In this paper, we propose a third-order accurate IMEX-RK scheme, whose coefficients come from the order conditions based on the Taylor expansion~\cite{2010Implicit}. Furthermore, we conduct an unconditional stability analysis which does not rely on the value of $\alpha$. More importantly, an improvement in efficiency and stability over the above-mentioned numerical methods will be clearly demonstrated. 

This paper is organized as follows. In Section~\ref{S:method}, we first present the LL model and give a brief introduction of the IMEX-RK schemes. The third-order IMEX-RK scheme is constructed in Section~\ref{Section:IMEX RK3 scheme} by the aid of order condition, and the stability condition of scheme proposed is proved from a theoretical point of view. Section~\ref{S:inequalities} is devoted to the related inequalities to facilitate the theoretical analysis. The convergence analysis and error estimate of the proposed IMEX-RK3 scheme is provided in Section~\ref{sec: convergence}. Accuracy tests are presented in Section~\ref{section:numercial tests} with a detailed check for the dependence on the artificial damping parameter. Finally, some concluding remarks are made in Section~\ref{section:conclusion}.

\section{The model and the proposed numerical method}\label{S:method}

\subsection{The Landau-Lifshitz equation}\label{model}
The dynamics of the magnetization in a ferromagnetic material occupying a bounded region $\Omega$ is governed by the Landau-Lifshitz, which reads as
\begin{align}
	&\boldsymbol{m}_{t}=-\boldsymbol{m} \times \boldsymbol{h}_{\mathrm{eff}}-\alpha \boldsymbol{m} \times\left(\boldsymbol{m} \times \boldsymbol{h}_{\mathrm{eff}}\right) ,\label{eq-1}\\
	&{{\left. \frac{\partial \boldsymbol m}{\partial \boldsymbol \nu } \right|}_{\Gamma}}=0.\label{eq-2}
\end{align}

In more details, consider the homogeneous Neumann boundary condition \eqref{eq-2}, $\Gamma=\partial \Omega$ and the magnetization $\boldsymbol m:\Omega \subset {\mathbb{R}^{d}}\to {\mathbb{S}^{2}},d=1,2,3$ is a 3-D vector field with $\left| \boldsymbol m \right| \equiv 1$. Here $\boldsymbol \nu$ is the unit outward normal vector along $\Gamma$. The first term of the right hand side of \eqref{eq-1} is the gyromagnetic term, while the
second term represents the damping term with a dimensionless damping parameter $\alpha>0$.

The effective field of a uniaxial material ${{\boldsymbol h}_{\text{eff}}}=-\frac{\delta F[\boldsymbol{m}]}{\delta \boldsymbol{m}}$ is computed from the free anergy functional
\begin{align*}
	F[\boldsymbol{m}]=\frac{\mu_{0} M_{s}^{2}}{2} \int_{\Omega}\left(\epsilon|\nabla \boldsymbol{m}|^{2}+Q\left(m_{2}^{2}+m_{3}^{2}\right)-\boldsymbol{h}_{s} \cdot \boldsymbol{m}-2 \boldsymbol{h}_{e} \cdot \boldsymbol{m}\right) \mathrm{d} \boldsymbol{x},
\end{align*}
corresponding to the exchange energy, the anisotropy energy, the magnetostatic energy, and the Zeeman energy parts, respectively. Here we have ${{\boldsymbol h}_{\text{eff}}}=\epsilon \Delta \boldsymbol m-Q({{m}_{2}}{{\boldsymbol e}_{2}}+{{m}_{3}}{{\boldsymbol e}_{3}})+{{\boldsymbol h}_{s}}+{{\boldsymbol h}_{e}}$, consists of the exchange field, the anisotropy field, the stray field ${{\boldsymbol h}_{\text{s}}}$, and the external field ${{\boldsymbol h}_{\text{e}}}$. In this formula, $Q={{K}_{u}}/({{\mu }_{0}}M_{s}^{2})$ and $\epsilon ={{C}_{ex}}/({{\mu }_{0}}M_{s}^{2}{{L}^{2}})$ are the dimensionless parameters with ${{C}_{ex}}$ the exchange constant, ${{K}_{u}}$ is the anisotropy constant, $L$ is the diameter of ferromagnetic body, ${{\mu }_{0}}$ is the permeability of vacuum, and $M_{s}$ stands for the saturation magnetization, respectively.
The two unit vectors are given by ${{\boldsymbol e}_{2}}={{(0,1,0)}^{T}}$, ${{\boldsymbol e}_{3}}={{(0,0,1)}^{T}}$, and the stray field ${{\boldsymbol h}_{\text{s}}}$ takes the form
\begin{align*}
	{{\boldsymbol h}_{s}}=-\nabla \int_{\Omega }{\nabla N(\boldsymbol  x-\boldsymbol y)\cdot \boldsymbol m(\boldsymbol y)\mathrm{d}\boldsymbol y},
\end{align*}
where $N(\boldsymbol x)=-\frac{1}{4\pi \left| \boldsymbol x \right|}$ is the Newtonian potential. The following notation is made to simplify the presentation: 
\begin{align}\label{lowerterm}
	\emph{ \textbf{f}}=-Q({{m}_{2}}{{\boldsymbol e}_{2}}+{{m}_{3}}{{\boldsymbol e}_{3}})+{{\boldsymbol h}_{s}}+{{\boldsymbol h}_{e}}.
\end{align}
Consequently, the LL equation \eqref{eq-1} could be reformulated as
\begin{align}
	{{\boldsymbol m}_{t}}=-\boldsymbol m\times (\epsilon \Delta \boldsymbol m+\emph{ \textbf{f}})-\alpha \boldsymbol m\times \boldsymbol m\times (\epsilon \Delta \boldsymbol m+\emph{ \textbf{f}}).
	\label{eq-3}
\end{align}
The LL equation has several equivalent forms. For instance, according to the formula 
\begin{equation}
	\boldsymbol a\times (\boldsymbol b\times \boldsymbol c)=(\boldsymbol a\cdot \boldsymbol c)\boldsymbol b-(\boldsymbol a\cdot \boldsymbol b)\boldsymbol c,   \qquad \boldsymbol a,\boldsymbol b, \boldsymbol c \in {{\mathbb{R}}^{3}},
\end{equation}
and by the fact of $|\boldsymbol{m}| \equiv 1$, an equivalent form could be deduced as follows 
\begin{align}
	\boldsymbol{m}_{t}=\alpha(\epsilon \Delta \boldsymbol{m}+\boldsymbol{f})+\alpha\left(\epsilon|\nabla \boldsymbol{m}|^{2}-\boldsymbol{m} \cdot \boldsymbol{f}\right) \boldsymbol{m}-\boldsymbol{m} \times(\epsilon \Delta \boldsymbol{m}+\boldsymbol{f}).
	\label{eq-4}
\end{align}

Some notations are needed in the numerical approximation. To ease the presentation, set $\Omega =( 0,1 )^d, d=1,2,3$, in which $d$ represents the dimension, and the final time is given by $T$. In the 1-D case, the domain $\Omega$ is divided into $N$ equal parts with $h=1/N$. In order to approximate the boundary condition \eqref{eq-2}, the ghost points are introduced as in Figure~\ref{fig1}, which displays a schematic picture of 1-D spatial grids, with ${{x}_{i-\frac{1}{2}}}=(i-\frac{1}{2})h,i=1,2,\cdots ,N$.
\begin{figure}[htbp]
	\centering
	\includegraphics[width=12cm,height=3cm]{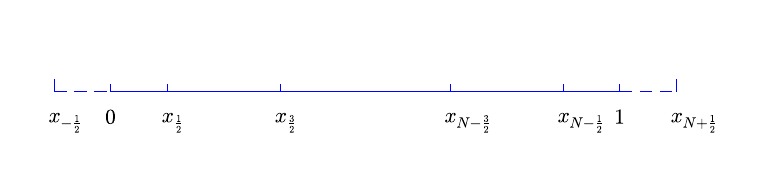}
	\caption{The 1-D spatial grids, where ${{x}_{-\frac{1}{2}}}$ and ${{x}_{N+\frac{1}{2}}}$ are two ghost points.}\label{fig1}
\end{figure}
The construction of the 3-D grid points is similar. For simplicity, we set $h_x=h_y=h_z = h,
{{\mathbf m}_{i,j,k}}=\mathbf m((i-\frac{1}{2})h,(j-\frac{1}{2})h,(k-\frac{1}{2})h), 0\le i,j,k\le N+1 $.
In the temporal discretization, we denote ${{t}^{n}}=nk$ with $k$ the step-size and $n\le \left[ \frac{T}{k} \right]$. Moreover, $\Delta_h \mathbf m$ represents the standard second-order centered difference stencil as
\begin{equation*}
	\begin{split}
		{{\Delta }_{h}}{{\mathbf m}_{i,j,k}} =& \frac{{{\mathbf m}_{i+1,j,k}}-2{{\mathbf m}_{i,j,k}}+{{\mathbf m}_{i-1,j,k}}}{{h^{2}}}\\
		&+\frac{{{\mathbf m}_{i,j+1,k}}-2{{\mathbf m}_{i,j,k}}+{{\mathbf m}_{i,j-1,k}}}{{h^{2}}}\\
		&+\frac{{{\mathbf m}_{i,j,k+1}}-2{{\mathbf m}_{i,j,k}}+{{\mathbf m}_{i,j,k-1}}}{{h^{2}}}.
	\end{split}
\end{equation*}
A second-order approximation to the Neumann boundary condition results in
\begin{eqnarray}
	\begin{aligned}
		& {{\mathbf m}_{0,j,k}}={{\mathbf m}_{1,j,k}}, \, \, \, {{\mathbf m}_{N,j,k}}={{\mathbf m}_{N+1,j,k}}, \, \, \, j,k=1,\cdots,N ,  \\
		& {{\mathbf m}_{i,0,k}}={{\mathbf m}_{i,1,k}}, \, \, \, {{\mathbf m}_{i,N,k}}={{\mathbf m}_{i,N+1,k}}, \, \, \, i,k=1,\cdots,N, \\
		& {{\mathbf m}_{i,j,0}}={{\mathbf m}_{i,j,1}}, \, \, \, {{\mathbf m}_{i,j,N}}={{\mathbf m}_{i,j,N+1}}, \, \, \, i,j=1,\cdots,N.
	\end{aligned}
	\label{BC-1}
\end{eqnarray}

\subsection{Implicit-explicit Runge-Kutta methods}\label{IMEX-RK}
For any time-dependent nonlinear equation, the key point of implicit-explicit (IMEX) numerical method relies on an implicit treatment of the dominant linear term and explicit treatment of the remaining terms~\cite{ascher1997implicit}. In fact, such an implicit treatment of the dominant linear term is necessary to ensure a numerical stability. However, the linear diffusion term does not dominate the magnetization dynamics in the LL equation, and thus a direct application of IMEX method is not appropriate. Motivated by this observation and the work in \cite{wang2020local}, a natural approach is to add an artificial diffusion term, then apply RK method to the time discretization.

Following this idea, we introduce an artificial Laplacian term $\beta\Delta \boldsymbol{m}$ into~\eqref{eq-3} and rewrite the LL equation as
\begin{align}
	{{\boldsymbol m}_{t}} = \underbrace{-\boldsymbol m\times (\epsilon \Delta \boldsymbol m+\emph{ \textbf{f}})-\alpha \boldsymbol m\times \boldsymbol m\times (\epsilon \Delta \boldsymbol m+\emph{ \textbf{f}}) -\beta   {{\Delta }} \boldsymbol m}_{{N(t,\boldsymbol m)}}+\underbrace{\beta   {{\Delta }}  \boldsymbol m}_{{L(t,\boldsymbol m)}},
	\label{eq-9}
\end{align}
in which the artificial term is denoted as $L(t,\boldsymbol m)$, and all the remaining terms are included in $N(t,\boldsymbol m)$.

An IMEX Runge-Kutta scheme consists of applying an implicit discretization to the linear term and an explicit computation of the nonlinear term. Its application to~\eqref{eq-9} takes the form
\begin{align}
	& {{\boldsymbol m}^{(i)}}={{\boldsymbol m}^{n}}-k\sum\limits_{j=1}^{i-1}{{{{\tilde{a}}}_{ij}}{{N}^{(j)}}(t,\boldsymbol m)}+k\sum\limits_{j=1}^{s}{{{a}_{ij}}{{L}^{(j)}}(t,\boldsymbol m)}, \\
	& {{\boldsymbol m}^{n+1}}={{\boldsymbol m}^{n}}-k\sum\limits_{i=1}^{s}{{{{\tilde{b}}}_{i}}{{N}^{(i)}}(t,\boldsymbol m)}+k\sum\limits_{i=1}^{s}{{{b}_{i}}{{L}^{(i)}}(t,\boldsymbol m)}.
	\label{eq-10}
\end{align}

In more details, the matrices $\tilde{A}=({{\tilde{a}}_{ij}})$, ${{\tilde{a}}_{ij}}=0$ for $j\ge i$ and $A=({{a}}_{ij})$ are $s\times s$ matrices such that the resulting algorithm is explicit in $N(t,\boldsymbol m)$ and implicit in $L(t,\boldsymbol m)$. An IMEX Runge-Kutta scheme is characterized by these two matrices and the coefficient vectors $\tilde{b}={{({{\tilde{b}}_{1}},{{\tilde{b}}_{2}},\cdots ,{{\tilde{b}}_{s}})}^{T}},b={{({{b}_{1}},{{b}_{2}},\cdots ,{{b}_{s}})}^{T}}$, with $k$ the step-size.

For the sake of simplicity and numerical implementation efficiency at each step, it is natural to consider diagonally implicit Runge-Kutta (DIRK) schemes \cite{2006Solving}.

The IMEX Runge-Kutta scheme can be represented by a double tableau in the usual Butcher notation,

\begin{minipage}[c]{0.5\textwidth}
	\centering
	\begin{tabular}{c|c}
		$\tilde{c}$ & $\tilde{A}$ \\ \hline & $\tilde{b}^T$
	\end{tabular}
\end{minipage}
\begin{minipage}[c]{0.25\textwidth}
	\centering
	\begin{tabular}{l|l}
		$c$ & $A$ \\
		\hline & $b^{\top}$
	\end{tabular}
\end{minipage}\\
where the coefficients $\tilde{c}$ and $c$ are given by the usual relation
\begin{align}
	{{\tilde{c}}_{i}}=\sum\limits_{j=1}^{i-1}{{{{\tilde{a}}}_{ij}}}, \quad
	{{c}_{i}}=\sum\limits_{j=1}^{i}{{{a}_{ij}}}.
	\label{eq-11}
\end{align}

Notice that this relation may not be necessary, it is just simple to use higher order Taylor expansion of the exact and numerical solution by the ``rooted trees" theory. An application of a DIRK scheme for $L(t,\boldsymbol m)$ is a sufficient condition to ensure that $N(t,\boldsymbol m)$ is always explicitly evaluated.

\section{Construction of IMEX RK3 scheme and its stability condition}\label{Section:IMEX RK3 scheme}
\subsection{Order conditions}

The general technique to derive order conditions for a Runge-Kutta method is based on the Taylor expansion of the exact and numerical solution; the relevant derivation and more details are referred to \cite{2010Implicit}. Here we give the order conditions for IMEX Runge-Kutta schemes, up to the third order accuracy. It is assumed that the coefficients ${{\tilde{c}}_{i}},{{c}_{i}},{{\tilde{a}}_{ij}},{{a}_{ij}}$ satisfy 
condition~\eqref{eq-11}. In turn, the order conditions are derived as the follows. \\
First order
\begin{align}
	\sum_{i=1}^{s} \tilde{b}_i=1, \quad \sum_{i=1}^s b_{i}=1.
	\label{eq-12}
\end{align}
Second order
\begin{align}
	& \sum_i \tilde{b}_i \tilde{c}_i=1 / 2, \quad \sum_i b_i c_i=1 / 2,
	\label{eq-13}
\end{align}
\begin{align}
	& \sum_i \tilde{b}_i c_i=1 / 2, \quad \sum_i b_i \tilde{c}_i=1 / 2 .
	\label{eq-14}
\end{align}
Third order
\begin{equation}
	\begin{aligned}
		& \sum_{i j} \tilde{b}_i {\tilde{a}}_{ij} \tilde{c}_j=1 / 6, \quad \sum_i \tilde{b}_i \tilde{c}_i \tilde{c}_i=1 / 3, \\
		& \sum_{i j} b_i {a}_{ij} c_j=1 / 6, \quad \sum_i b_i c_i c_i=1 / 3 , 
		\label{eq-15}
	\end{aligned}
\end{equation}
\begin{equation}
	\begin{aligned}
		\sum_{i j} \tilde{b}_i {\tilde{a}}_{ij} c_j=1 / 6, \quad \sum_{i j} \tilde{b}_i {a}_{ij} \tilde{c}_j=1 / 6, \quad \sum_{i j} \tilde{b}_i {a}_{ij} c_j=1 / 6,\\
		\sum_{i j} b_i {\tilde{a}}_{ij} c_j=1 / 6, \quad \sum_{i j} b_i {a}_{ij} \tilde{c}_j=1 / 6, \quad \sum_{i j} {b}_i {\tilde{a}}_{ij} \tilde{c}_j=1 / 6,
		\label{eq-16}
	\end{aligned}
\end{equation}
$$
\begin{aligned}
	\sum_i \tilde{b}_i c_i c_i=1 / 3, \quad \sum_i \tilde{b}_i \tilde{c}_i c_i=1 / 3, \quad \sum_i b_i \tilde{c}_i \tilde{c}_i=1 / 3, \quad \sum_i b_i \tilde{c}_i c_i=1 / 3.
\end{aligned}
$$

Conditions~\eqref{eq-12}, \eqref{eq-13}~and~\eqref{eq-15} are the standard order conditions for the two tableau, conditions \eqref{eq-14} and \eqref{eq-16} arise because of the coupled nature of the RK algorithms.

\begin{remark}
	The order conditions will be greatly simplified if ${{\tilde{c}}_{i}}={{c}_{i}}$, and ${{\tilde{b}}_{i}}={{b}_{i}}$, i.e., these two tableau only differ by the value of the coefficients matrices. Because of this fact, the standard conditions are enough to guarantee that the combined scheme could achieve the corresponding order. It is noteworthy that this is only true for the RK schemes up to the third order accuracy.
\end{remark}

\subsection{Construction of IMEX RK3 scheme}
In this section, we focus on the construction of a third-order IMEX Runge-Kutta scheme, which satisfies the stability condition in our framework. For simplicity of presentation, we consider the case of ${{\tilde{c}}_{i}}={{c}_{i}}$ and ${{\tilde{b}}_{i}}={{b}_{i}}$. The discussion of a general method may be more complicated, while the formulation is similar.

For the convenience, we list the Butcher tableau with undetermined coefficients:  
\begin{equation}\label{eqn:RK3}
	\begin{array}{c|ccccc|ccccc}
		\rule{0pt}{20pt}
		0 & 0 & 0 & 0 & 0 & 0 & 0 & 0 & 0 & 0 & 0 \\
		{{c}_{2}} & 0 & {{a}_{22}} & 0 & 0 & 0 & {{\tilde{a}}_{21}} & 0 & 0 & 0 & 0 \\
		{{c}_{3}}& 0 & {{a}_{32}} & {{a}_{33}} & 0 & 0 & {{\tilde{a}}_{31}} &{{\tilde{a}}_{32}} & 0 & 0 & 0 \\
		{{c}_{4}} & 0 & {{a}_{42}} & {{a}_{43}} & {{a}_{44}} & 0 & {{\tilde{a}}_{41}} & {{\tilde{a}}_{42}} & {{\tilde{a}}_{43}} & 0 & 0 \\
		\hline & 0 & {{b}_{2}} & {{b}_{3}}& {{b}_{4}} & 0 & 0 & {{b}_{2}} & {{b}_{3}}& {{b}_{4}} & 0
	\end{array}
\end{equation}
and the associated relation gives
\begin{equation}
	\begin{aligned}
		& {{c}_{2}}={{a}_{22}}={{{\tilde{a}}}_{21}}, \\
		& {{c}_{3}}={{a}_{32}}+{{a}_{33}}={{{\tilde{a}}}_{31}}+{{{\tilde{a}}}_{32}}, \\
		&{{c}_{4}}={{a}_{42}}+{{a}_{43}}+{{a}_{44}}={{{\tilde{a}}}_{41}}+{{{\tilde{a}}}_{42}}+{{{\tilde{a}}}_{43}}.
		\label{eq-17}
	\end{aligned}
\end{equation}

Begin with the standard order condition~\eqref{eq-15}, the unknowns also need to satisfy these equalities to ensure that a numerical stability in our framework:
\begin{equation}
	\begin{aligned}
		&{{b}_{2}}+{{b}_{3}}+{{b}_{4}}=1,\\
		&{{b}_{2}}{{c}_{2}}+{{b}_{3}}{{c}_{3}}+{{b}_{4}}{{c}_{4}}=1/2,\\
		&{{b}_{3}}{{{\tilde{a}}}_{32}}{{c}_{2}}+{{b}_{4}}({{{\tilde{a}}}_{42}}{{c}_{2}}+{{{\tilde{a}}}_{43}}{{c}_{3}})  =1/6, \\
		& {{b}_{2}}c_{2}^{2}+{{b}_{3}}c_{3}^{2}+{{b}_{4}}c_{4}^{2}=1/3, \\
		&{{b}_{2}}c_{2}^{2}+{{b}_{3}}{{a}_{32}}{{c}_{2}}+{{b}_{3}}{{a}_{33}}{{c}_{3}}+{{b}_{4}}({{a}_{42}}{{c}_{2}}
		+{{a}_{43}}{{c}_{3}}+{{a}_{44}}{{c}_{4}})=1/6.
		\label{eq-18}
	\end{aligned}
\end{equation}
Accordingly, at time step $t_n$, the corresponding marching algorithm method becomes 
\begin{equation}\label{IMEX-RK3-general}
	\left\{\begin{array}{l}
		\boldsymbol{{\tilde{m}}}_{1}=\boldsymbol{{m}}_{n}, \\ {\boldsymbol{\tilde{m}}_{2}}=\boldsymbol{{\tilde{m}}}_{1}+k \left({{\tilde{a}}_{21}}N(t_{n}^{0}, {\boldsymbol{\tilde{m}}_{1}})+{{a}_{22}}L(t_{n}^{1}, \boldsymbol{{\tilde{m}}}_{2})\right), \\ {\boldsymbol{\tilde{m}}_{3}}=\boldsymbol{{\tilde{m}}}_{1}+k \left({{\tilde{a}}_{31}}N(t_{n}^{0}, {\boldsymbol{\tilde{m}}_{1}})+{{\tilde{a}}_{32}}N(t_{n}^{1}, {\boldsymbol{\tilde{m}}_{2}})\right)\\
		\qquad \quad +k \left({{a}_{32}}L(t_{n}^{1}, \boldsymbol{{\tilde{m}}}_{2})+{{a}_{33}}L(t_{n}^{2}, \boldsymbol{{\tilde{m}}}_{3})\right),\\
		{\boldsymbol{\tilde{m}}_{4}}=\boldsymbol{{\tilde{m}}}_{1}+k \left({{\tilde{a}}_{41}}N(t_{n}^{0}, {\boldsymbol{\tilde{m}}_{1}})+{{\tilde{a}}_{42}}N(t_{n}^{1}, {\boldsymbol{\tilde{m}}_{2}}+{{\tilde{a}}_{43}}N(t_{n}^{2}, {\boldsymbol{\tilde{m}}_{3}})\right)\\
		\qquad \quad +k \left({{a}_{42}}L(t_{n}^{1}, \boldsymbol{{\tilde{m}}}_{2})+{{a}_{43}}L(t_{n}^{2}, \boldsymbol{{\tilde{m}}}_{3})+{{a}_{44}}L(t_{n}^{3}, \boldsymbol{{\tilde{m}}}_{4})\right),\\
		\boldsymbol{{m}}_{n+1}=\boldsymbol{{\tilde{m}}}_{1}+k \left({{b}_{2}}N(t_{n}^{1}, {\boldsymbol{\tilde{m}}_{2}})+{{b}_{3}}N(t_{n}^{2}, {\boldsymbol{\tilde{m}}_{3}})+{{b}_{4}}N(t_{n}^{3}, {\boldsymbol{\tilde{m}}_{4}})\right)\\
		\qquad \quad +k \left({{b}_{2}}L(t_{n}^{1}, \boldsymbol{{\tilde{m}}}_{2})+{{b}_{3}}L(t_{n}^{2}, \boldsymbol{{\tilde{m}}}_{3})+{{b}_{4}}L(t_{n}^{3}, \boldsymbol{{\tilde{m}}}_{4})\right).
	\end{array}\right.
\end{equation}

\begin{remark}
	In the above formulation, by the aid of a fully explicit treatment for the nonlinear parts and implicit treatment for the linear part, the resulting numerical method only requires a standard Poisson solver at each time step. This fact will greatly reduce the computational cost, since the $\mathrm{FFT}$ fast solver could be efficiently applied, due to the constant coefficient $\mathrm{SPD}$ structure of the involved linear system.
\end{remark}

In a simple case with only linear diffusion term, we denote $L_h = \beta \Delta_h$ and take ${{b}_{2}}={{a}_{42}}$, ${{b}_{3}}={{a}_{43}}$, ${{b}_{4}}={{a}_{44}}$. The IMEX-RK3 scheme~\eqref{IMEX-RK3-general} is represented as
\begin{equation}\label{IMEX-RK3-Linear}
	\left\{\begin{array}{l}
		\boldsymbol{{\tilde{m}}}_{1}=\boldsymbol{{m}}_{n}, \\ {\boldsymbol{\tilde{m}}_{2}}=\boldsymbol{{\tilde{m}}}_{1}+k\left({{a}_{22}} L_h({\boldsymbol{\tilde{m}}_{2}})\right), \\
		{\boldsymbol{\tilde{m}}_{3}}=\boldsymbol{{\tilde{m}}}_{1}
		+k \left({{a}_{32}}L_h({\boldsymbol{\tilde{m}}_{2}})+{{a}_{33}}L_h({\boldsymbol{\tilde{m}}_{3}})\right),\\
		{\boldsymbol{\tilde{m}}_{4}}=\boldsymbol{{\tilde{m}}}_{1}
		+k \left({{a}_{42}}L_h({\boldsymbol{\tilde{m}}_{2}})+{{a}_{43}}L_h({\boldsymbol{\tilde{m}}_{3}})
		+{{a}_{44}}L_h({\boldsymbol{\tilde{m}}_{4}})\right),\\
		\boldsymbol{{m}}_{n+1}=\tilde{\boldsymbol m}_4.\\
	\end{array}\right.
\end{equation}

\subsection{Stability condition}
To facilitate the stability analysis of the method proposed above, the numerical system is rewritten as
\begin{align}
	&
	\boldsymbol{{\tilde{m}}}_{1}=\boldsymbol{{m}}_{n},
	\label{IMEX-RK3-Linear-1-1}
	\\
	&
	\frac{\tilde{\boldsymbol m}_2 - \boldsymbol m_n}{k} = {{a}_{22}}\beta \Delta_h \tilde{\boldsymbol m}_2 ,   \label{IMEX-RK3-Linear-1-2}
	\\
	&
	\frac{\tilde{\boldsymbol m}_3 - \boldsymbol m_n}{k} = {{a}_{32}}\beta  \Delta_h \tilde{\boldsymbol m}_2
	+ {{a}_{33}}\beta  \Delta_h \tilde{\boldsymbol m}_3 ,   \label{IMEX-RK3-Linear-1-3}
	\\
	&
	\frac{\tilde{\boldsymbol m}_4 - \boldsymbol m_n}{k} = {{a}_{42}}\beta  \Delta_h \tilde{\boldsymbol m}_2
	+ {{a}_{43}}\beta  \Delta_h \tilde{\boldsymbol m}_3+ {{a}_{44}}\beta  \Delta_h \tilde{\boldsymbol m}_4,   \label{IMEX-RK3-Linear-1-4}\\
	& \boldsymbol m_{n+1} = \tilde{\boldsymbol m}_4 . \label{IMEX-RK3-Linear-1-5}
\end{align}

In fact, the numerical stability of the Runge-Kutta algorithm could be demonstrated by subtracting~\eqref{IMEX-RK3-Linear-1-1} from \eqref{IMEX-RK3-Linear-1-2}, \eqref{IMEX-RK3-Linear-1-2} from \eqref{IMEX-RK3-Linear-1-3}, \eqref{IMEX-RK3-Linear-1-3} from \eqref{IMEX-RK3-Linear-1-4}. As a consequence, the following equivalent numerical system is obtained:
\begin{align}
	&
	\boldsymbol{{\tilde{m}}}_{1}=\boldsymbol{{m}}_{n},
	\label{IMEX-RK3-Linear-2-1}
	\\
	&
	\frac{\tilde{\boldsymbol m}_2 - \tilde{\boldsymbol m}_1}{k} = {{a}_{22}}\beta \Delta_h \tilde{\boldsymbol m}_2 ,   \label{IMEX-RK3-Linear-2-2}
	\\
	&
	\frac{\tilde{\boldsymbol m}_3 - \tilde{\boldsymbol m}_2}{k} = ({{a}_{32}}-{{a}_{22}})\beta \Delta_h \tilde{\boldsymbol m}_2+
	{{a}_{33}}\beta \Delta_h \tilde{\boldsymbol m}_3,   \label{IMEX-RK3-Linear-2-3}
	\\
	&
	\frac{\tilde{\boldsymbol m}_4 - \tilde{\boldsymbol m}_3}{k} = ({{a}_{42}}-{{a}_{32}})\beta \Delta_h \tilde{\boldsymbol m}_2+
	({{a}_{43}}-{{a}_{33}})\beta \Delta_h \tilde{\boldsymbol m}_3 + {{a}_{44}}\beta \Delta_h \tilde{\boldsymbol m}_4 ,   \label{IMEX-RK3-Linear-2-4}
	\\
	& \boldsymbol m_{n+1} = \tilde{\boldsymbol m}_4 . \label{IMEX-RK3-Linear-2-5}
\end{align}

In the first step, taking a discrete inner product with~\eqref{IMEX-RK3-Linear-2-2} by $2 \tilde{\boldsymbol m}_2$ yields 
\begin{equation}
	\| \tilde{\boldsymbol m}_2 \|_2^2 - \| \boldsymbol m_n \|_2^2 + \| \tilde{\boldsymbol m}_2 - \boldsymbol m_n \|_2^2
	+ 2{{a}_{22}}\beta k \| \nabla_h \tilde{\boldsymbol m}_2 \|_2^2 = 0 ,  \label{stability-1}
\end{equation}
with an application of the summation-by-parts formula. Similarly, taking a discrete inner product with~\eqref{IMEX-RK3-Linear-2-3} by $2 \tilde{\boldsymbol m}_3$,  with~\eqref{IMEX-RK3-Linear-2-4} by $2 \tilde{\boldsymbol m}_4$, turns out to be
\begin{equation}
	\begin{aligned}
		&
		\| \tilde{\boldsymbol m}_3 \|_2^2 - \| \tilde{\boldsymbol m}_2 \|_2^2 + \| \tilde{\boldsymbol m}_3 - \tilde{\boldsymbol m}_2 \|_2^2
		+ 2{{a}_{33}}\beta k \| \nabla_h \tilde{\boldsymbol m}_3 \|_2^2
		\\
		&
		= 2({{a}_{22}}-{{a}_{32}})\beta k \langle \nabla_h \tilde{\boldsymbol m}_2 ,  \nabla_h \tilde{\boldsymbol m}_3 \rangle ,  \label{stability-2}
	\end{aligned}
\end{equation}
\begin{equation}
	\begin{aligned}
		&
		\| \tilde{\boldsymbol m}_4 \|_2^2 - \| \tilde{\boldsymbol m}_3 \|_2^2 + \| \tilde{\boldsymbol m}_4 - \tilde{\boldsymbol m}_3 \|_2^2+ 2{{a}_{44}}\beta k \| \nabla_h \tilde{\boldsymbol m}_4 \|_2^2
		\\
		&
		= 2({{a}_{32}}-{{a}_{42}})\beta k \langle \nabla_h \tilde{\boldsymbol m}_2 ,  \nabla_h \tilde{\boldsymbol m}_4 \rangle + 2({{a}_{33}}-{{a}_{43}})\beta k \langle \nabla_h \tilde{\boldsymbol m}_3 ,  \nabla_h \tilde{\boldsymbol m}_4 \rangle. \label{stability-3}
	\end{aligned}
\end{equation}
Subsequently, a summation of~\eqref{stability-1}--\eqref{stability-3} gives 
\begin{equation}
	\begin{aligned}
		&
		\| \tilde{\boldsymbol m}_4 \|_2^2 - \| \boldsymbol m_n \|_2^2 + \| \tilde{\boldsymbol m}_2 - \boldsymbol m_n \|_2^2
		+ \| \tilde{\boldsymbol m}_3 - \tilde{\boldsymbol m}_2 \|_2^2
		\\
		&
		+ \| \tilde{\boldsymbol m}_4 - \tilde{\boldsymbol m}_3 \|_2^2
		+ 2{{a}_{22}}\beta k \| \nabla_h \tilde{\boldsymbol m}_2 \|_2^2
		+ 2{{a}_{33}}\beta k \| \nabla_h \tilde{\boldsymbol m}_3 \|_2^2
		\\
		&
		+ 2{{a}_{44}}\beta k \| \nabla_h \tilde{\boldsymbol m}_4 \|_2^2  = 2({{a}_{22}}-{{a}_{32}})\beta k \langle \nabla_h \tilde{\boldsymbol m}_2 ,  \nabla_h \tilde{\boldsymbol m}_3 \rangle
		\\
		&
		+2({{a}_{32}}-{{a}_{42}})\beta k \langle \nabla_h \tilde{\boldsymbol m}_2 ,  \nabla_h \tilde{\boldsymbol m}_4 \rangle+ 2({{a}_{33}}-{{a}_{43}})\beta k \langle \nabla_h \tilde{\boldsymbol m}_3 ,  \nabla_h \tilde{\boldsymbol m}_4 \rangle.
		\label{stability-4}
	\end{aligned}
\end{equation}
In turn, an application of Cauchy inequality reveals that 
\begin{align}
	&
	2({{a}_{22}}-{{a}_{32}})\beta k \langle \nabla_h \tilde{\boldsymbol m}_2 ,  \nabla_h \tilde{\boldsymbol m}_3 \rangle
	\le \left| {{a}_{32}}-{{a}_{22}} \right| \beta k ( \| \nabla_h \tilde{\boldsymbol m}_2 \|_2^2 + \|  \nabla_h \tilde{\boldsymbol m}_3  \|_2^2 ) ,
	\label{stability-5-1}
	\\
	&
	2({{a}_{32}}-{{a}_{42}})\beta k \langle \nabla_h \tilde{\boldsymbol m}_2 ,  \nabla_h \tilde{\boldsymbol m}_4 \rangle
	\le \left| {{a}_{42}}-{{a}_{32}} \right| \beta k ( \| \nabla_h \tilde{\boldsymbol m}_2 \|_2^2 + \|  \nabla_h \tilde{\boldsymbol m}_4  \|_2^2 ) ,
	\label{stability-5-2}
	\\
	&
	2({{a}_{33}}-{{a}_{43}})\beta k \langle \nabla_h \tilde{\boldsymbol m}_3 ,  \nabla_h \tilde{\boldsymbol m}_4 \rangle
	\le \left| {{a}_{43}}-{{a}_{33}} \right| \beta k ( \| \nabla_h \tilde{\boldsymbol m}_3 \|_2^2 + \|  \nabla_h \tilde{\boldsymbol m}_4  \|_2^2 ).
	\label{stability-5-3}
\end{align}
Going back~\eqref{stability-4}, we arrive at 
\begin{equation}
	\begin{aligned}
		&
		\| \tilde{\boldsymbol m}_4 \|_2^2 - \| \boldsymbol m_n \|_2^2 + \| \tilde{\boldsymbol m}_2 - \boldsymbol m_n \|_2^2
		+ \| \tilde{\boldsymbol m}_3 - \tilde{\boldsymbol m}_2 \|_2^2 + \| \tilde{\boldsymbol m}_4 - \tilde{\boldsymbol m}_3 \|_2^2
		\\
		&
		+ 2{{a}_{22}}\beta k \| \nabla_h \tilde{\boldsymbol m}_2 \|_2^2 + 2{{a}_{33}}\beta k \| \nabla_h \tilde{\boldsymbol m}_3 \|_2^2 + 2{{a}_{44}}\beta k \| \nabla_h \tilde{\boldsymbol m}_4 \|_2^2
		\\
		&
		\le \left| {{a}_{32}}-{{a}_{22}} \right| \beta k (\| \nabla_h \tilde{\boldsymbol m}_2 \|_2^2 + \|  \nabla_h \tilde{\boldsymbol m}_3  \|_2^2 )+\left| {{a}_{42}}-{{a}_{32}} \right| \beta k
		\\
		&
		( \| \nabla_h \tilde{\boldsymbol m}_2 \|_2^2 + \|  \nabla_h \tilde{\boldsymbol m}_4  \|_2^2 )+ \left| {{a}_{43}}-{{a}_{33}} \right| \beta k( \| \nabla_h \tilde{\boldsymbol m}_3 \|_2^2 + \|  \nabla_h \tilde{\boldsymbol m}_4  \|_2^2 ).
		\label{stability-5}
	\end{aligned}
\end{equation}
Of course, a careful simplification reveals that
\begin{equation}
	\begin{aligned}
		&
		\| \tilde{\boldsymbol m}_4 \|_2^2 - \| \boldsymbol m_n \|_2^2 + \| \tilde{\boldsymbol m}_2 - \boldsymbol m_n \|_2^2
		+ \| \tilde{\boldsymbol m}_3 - \tilde{\boldsymbol m}_2 \|_2^2+ \| \tilde{\boldsymbol m}_4 - \tilde{\boldsymbol m}_3 \|_2^2+ (2{{a}_{22}}
		\\
		&
		- \left| {{a}_{32}}-{{a}_{22}} \right| - \left| {{a}_{42}}-{{a}_{32}} \right| ) \beta k  \| \nabla_h \tilde{\boldsymbol m}_2 \|_2^2+ (2{{a}_{33}} - \left| {{a}_{32}}-{{a}_{22}} \right|- \left| {{a}_{43}}-{{a}_{33}} \right|)
		\\
		&
		\beta k  \| \nabla_h \tilde{\boldsymbol m}_3 \|_2^2+(2{{a}_{44}} - \left| {{a}_{42}}-{{a}_{32}} \right| - \left| {{a}_{43}}-{{a}_{33}} \right|)\beta k  \| \nabla_h \tilde{\boldsymbol m}_4 \|_2^2
		\le 0.
	\end{aligned}
\end{equation}
Because of the fact that $\boldsymbol m_{n+1} = \tilde{\boldsymbol m}_4$, an $\ell^\infty (0, T; \ell^2) \cap \ell^2 (0, T; H_{h}^{1})$ bound of the numerical solution could be derived for the IMEX-RK3 scheme, under the stability condition as follows:
\begin{equation}\label{IMEX-RK3-stability}
	\left\{\begin{array}{l}
		2{{a}_{22}} - (\left| {{a}_{32}} - {{a}_{22}} \right| + \left| {{a}_{42}}-{{a}_{32}} \right|  ) > 0, \\
		2{{a}_{33}} - (\left| {{a}_{32}} - {{a}_{22}} \right| + \left| {{a}_{43}}-{{a}_{33}} \right|) > 0,\\
		2{{a}_{44}} - (\left| {{a}_{42}} - {{a}_{32}} \right| + \left| {{a}_{43}}-{{a}_{33}} \right| ) > 0 ,
	\end{array}\right.
\end{equation}
which is equivalent to
\begin{equation}\label{IMEX-RK3-stability-2}
	\left\{\begin{array}{l}
		2{{a}_{22}} >  \left| {{a}_{32}} - {{a}_{22}} \right| + \left| {{a}_{42}}-{{a}_{32}} \right|,    \\
		2{{a}_{33}} > \left| {{a}_{32}} - {{a}_{22}} \right| + \left| {{a}_{43}}-{{a}_{33}} \right|,    \\
		2{{a}_{44}} > \left| {{a}_{42}} - {{a}_{32}} \right| + \left| {{a}_{43}}-{{a}_{33}} \right|.
	\end{array}\right.
\end{equation}

We are interested in whether there is a third order IMEX Runge-Kutta method that simultaneously satisfies \eqref{IMEX-RK3-stability-2} and \eqref{eq-18}. It is obvious that there are infinite set of solutions corresponding to the linear part and nonlinear part through numerical calculation. Here we provide a Butcher tableau which satisfies these conditions.

	\begin{scriptsize}
		\begin{equation*}\label{eqn:RK3-2}
			\begin{array}{c|cccc|cccc}
				\rule{0pt}{20pt}
				0 & 0 & 0 & 0 & 0  & 0 & 0 & 0 & 0  \\
				0.62500000 & 0 & 0.62500000 & 0 & 0  & 0.62500000 & 0 & 0 & 0 \\
				0.31347352& 0 & -0.23587004 & 0.54934357 & 0  & 0.17055712 & 0.14291640 & 0 & 0  \\
				1 & 0 & 0.08500000 & 0.68187464 & 0.23312535 & 0 & 0.45000000 & 0.55000000 & 0  \\
				\hline & 0 & 0.08500000 & 0.68187464 & 0.23312535 & 0 & 0.08500000 & 0.68187464 & 0.23312535 
			\end{array}
		\end{equation*}
	\end{scriptsize}

Consequently, the marching algorithm in IMEX-RK3 at time step $t_{n}$ becomes 
		\begin{equation}\label{IMEX-RK3}
			\left\{\begin{array}{l}
				\boldsymbol{{\tilde{m}}}_{1}=\boldsymbol{{m}}_{n}, \\ {\boldsymbol{\tilde{m}}_{2}}=\boldsymbol{{\tilde{m}}}_{1}+0.62500000kN(\boldsymbol{{\tilde{m}}}_{1})
				+0.62500000kL(\boldsymbol{{\tilde{m}}}_{2}), \\ {\boldsymbol{\tilde{m}}_{3}}=\boldsymbol{{\tilde{m}}}_{1}+0.17055712kN(\boldsymbol{{\tilde{m}}}_{1})
				+0.14291640kN(\boldsymbol{{\tilde{m}}}_{2})\\
				\qquad\quad-0.23587004kL(\boldsymbol{{\tilde{m}}}_{2})+0.54934357kL(\boldsymbol{{\tilde{m}}}_{3}),\\
				{\boldsymbol{\tilde{m}}_{4}}=\boldsymbol{{\tilde{m}}}_{1}+0.45000000kN(\boldsymbol{{\tilde{m}}}_{2})
				+0.55000000kN(\boldsymbol{{\tilde{m}}}_{3})\\
				\qquad \quad +0.08500000kL(\boldsymbol{{\tilde{m}}}_{2})+0.68187464kL(\boldsymbol{{\tilde{m}}}_{3})+0.23312535kL(\boldsymbol{{\tilde{m}}}_{4}),\\
				\boldsymbol{{m}}_{n+1}=\boldsymbol{{\tilde{m}}}_{1}+0.08500000kN(\boldsymbol{{\tilde{m}}}_{2})+0.68187464kN(\boldsymbol{{\tilde{m}}}_{3})+0.23312535kN(\boldsymbol{{\tilde{m}}}_{4})\\
				\qquad \quad+0.08500000kL(\boldsymbol{{\tilde{m}}}_{2})+0.68187464kL(\boldsymbol{{\tilde{m}}}_{3})+0.23312535kL(\boldsymbol{{\tilde{m}}}_{4}).
			\end{array}\right.
		\end{equation}

\section{Some preliminary inequalities}\label{S:inequalities}

In this section, some preliminary inequalities are reviewed, which will be useful in the error analysis presented in the next section. Proofs of the standard inverse inequality and discrete Gronwall inequality can be obtained in existing textbooks and references, see \cite{chen16, chenW20a, Ciarlet1978,Girault1986}; here the results are just cited.

\begin{definition} [$\ell^2$ inner product, ${{\left\| \cdot  \right\|}_{2}}$ norm]
	For grid functions ${{\boldsymbol f}_{h}}$ and ${{\boldsymbol g}_{h}}$ that take values on a uniform numerical grid, we define
	\begin{equation*}
		\left\langle\boldsymbol{f}_{h}, \boldsymbol{g}_{h}\right\rangle=h^{d} \sum_{{I} \in \Lambda_{d}} \boldsymbol{f}_{{I}} \cdot \boldsymbol{g}_{{I}} ,
	\end{equation*}
	where $\Lambda_{d}$ is the set of grid point, and ${I}$ is an index.
\end{definition}

In turn, the ${{\left\| \cdot  \right\|}_{2}}$ norm turns out to be 
\begin{equation*}
	\left\|\boldsymbol{f}_{h}\right\|_{2}=\left(\left\langle\boldsymbol{f}_{h}, \boldsymbol{f}_{h}\right\rangle\right)^{1 / 2}.
\end{equation*}
Furthermore, the discrete ${H^{1}}$ norm is introduced as
\begin{equation*}
	\left\|\boldsymbol{f}_{h}\right\|_{H^{1}}^{2}:=\left\|\boldsymbol{f}_{h}\right\|_{2}^{2}+\left\|\nabla_{h} \boldsymbol{f}_{h}\right\|_{2}^{2}.
\end{equation*}

\begin{definition}[ The discrete ${{\left\| \cdot  \right\|}_{\infty}}$ and ${{\left\| \cdot  \right\|}_{p}}$ norms]
	For grid functions ${{\boldsymbol f}_{h}}$ that take values on a uniform numerical grid, we define
	\begin{equation*}
		\left\|\boldsymbol f_{h}\right\|_{\infty}=\max _{{I} \in \Lambda_{d}}\left\|\boldsymbol{f}_{{I}}\right\|_{\infty}, \quad\left\|\boldsymbol{f}_{h}\right\|_{p}=\left(h^{d} \sum_{I \in \Lambda_{d}}\left|\boldsymbol{f}_{{I}}\right|^{p}\right)^{\frac{1}{p}},
		\quad 1 \leq p<+\infty .
	\end{equation*}
\end{definition}

\begin{lemma}[Summation by parts]\label{summation}	
	For any grid functions $\boldsymbol f_h$ and $\boldsymbol g_h$, with $\boldsymbol f_h$ satisfying the discrete boundary condition~\eqref{eq-2}, the following identity is valid:
	\begin{align} \label{sum1}
		\left\langle -\Delta_h \boldsymbol f_h,\boldsymbol g_h\right\rangle = \left\langle \nabla_h \boldsymbol f_h,\nabla_h \boldsymbol g_h\right\rangle .
	\end{align}
\end{lemma}

\begin{lemma} [Inverse inequality] \cite{chen16, chenW20a, Ciarlet1978} \label{ccclemC1}.
	For each vector-valued grid function $\boldsymbol f_h\in X$, we have
	\begin{align}
		&
		\|{\boldsymbol f}_h \|_{\infty} \leq \gamma {h}^{-1/2 }
		( \|{\boldsymbol f}_h \|_2 + \| \nabla_h \boldsymbol f_h \|_2 ) ,   \label{inverse-1} 	
		\\
		& 	
		\| \boldsymbol f_h \|_q \leq \gamma {h}^{-(\frac32 - \frac{3}{q}) } \| \boldsymbol f_h \|_2 ,  \, \, \, \forall 2 < q \le + \infty ,
		\label{inverse-2}	
	\end{align}
	in which constant $\gamma$ depends on $\Omega$, as well as the form of the discrete $\| \cdot \|_2$ norm.
\end{lemma}

\begin{lemma} [Discrete Gronwall inequality] \cite{Girault1986}\label{lem: Gronwall-1}.
	Let ${{\left\{ \left. {{\alpha }_{j}} \right\} \right.}_{j\ge 0}}, {{\left\{ \left. {{\beta }_{j}} \right\} \right.}_{j\ge 0}}$, and ${{\left\{ \left. {{\omega }_{j}} \right\} \right.}_{j\ge 0}}$ be sequences of real numbers such that
	and  we have the following discrete estimate:
	\begin{align}
		{{\alpha }_{j}}\le {{\alpha }_{j+1}},{{\beta }_{j}}\ge 0 \label{Gronwall-1}~~~~ and~~~~ {{\omega }_{j}}\le {{\alpha }_{j}}+\sum\limits_{i=0}^{j-1}{{{\beta }_{i}}{{\omega }_{i}}},\forall j\ge 0.
	\end{align}
	Then it holds that
	\begin{align}
		{{\omega }_{j}}\le {{\alpha }_{j}}\exp \left\{ \left. \sum\limits_{i=0}^{j-1}{{{\beta }_{i}}} \right\} \right.,\forall j\ge 0.
	\end{align}
\end{lemma}

\section{Convergence analysis for the proposed IMEX-RK3 scheme} \label{sec: convergence}

A theoretical analysis of any third order accurate, IMEX-RK scheme is very challenging, which comes from its multi-stage nature and highly complex nonlinear terms in the vector form. For the sake of simplicity, we will focus on the IMEX-RK3 scheme~\eqref{IMEX-RK3}. Note that a simplified nonlinear LL equation is considered, and only the damping term is considered. The convergence analysis of the proposed method is provided in this section.

\subsection{Convergence analysis of the IMEX-RK3 scheme~\eqref{IMEX-RK3} for the nonlinear LL equation}
Taking a simplified nonlinear LL equation~\eqref{eq-1} into consideration in this part, in which the gyromagnetic term is skipped:
\begin{align}
	{{\boldsymbol m}_{t}}= - \alpha \boldsymbol m \times ( \boldsymbol m \times ( \epsilon \Delta \boldsymbol m + \boldsymbol f ) ).
	\label{equation-LL-mod}
\end{align}
For any vector function $\boldsymbol m$ with $| \boldsymbol m | \equiv 1$, the nonlinear term $N (\boldsymbol m)$ could always be rewritten as follows, with a notation of $\beta = \alpha \epsilon$: 
\begin{equation}
	\begin{aligned}
		N (\boldsymbol m) = & - \alpha \boldsymbol m \times ( \boldsymbol m \times ( \epsilon \Delta \boldsymbol m + \boldsymbol f ) ) - \beta \Delta \boldsymbol m
		\\
		= & \beta ( \Delta \boldsymbol m + | \nabla \boldsymbol m |^2 \boldsymbol m ) - \alpha \boldsymbol m \times ( \boldsymbol m \times \boldsymbol f )
		- \beta \Delta \boldsymbol m
		\\
		= &
		\beta | \nabla \boldsymbol m |^2 \boldsymbol m  - \alpha \boldsymbol m \times ( \boldsymbol m \times \boldsymbol f ).
	\end{aligned}
	\label{LL-reformulation-2}
\end{equation}
On the other hand, the notation ${ A}_h\nabla_h$ stands for the second approximation to the gradient operator. In fact, it is an average gradient operator defined for the gird function $\boldsymbol m=(u_h, v_h, w_h)^T\in \boldsymbol X$, as ${A}_h\nabla_h\boldsymbol m_h=\nabla_h{A}_h\boldsymbol m_h$ and ${A}_h \boldsymbol m =({A}_xu_h,{A}_y v_h,{A}_z w_h)$:
\begin{equation*}
	{A}_x u_{i,j,\ell}=\frac{u_{i,j,\ell}+u_{i-1,j,\ell}}{2},\,
    {A}_y v_{i,j,\ell}=\frac{v_{i,j,\ell}+v_{i,j-1,\ell}}{2},\,
	{A}_z w_{i,j,\ell}=\frac{w_{i,j,\ell}+w_{i,j,\ell-1}}{2}.
\end{equation*}
Accordingly, the discrete form of the nonlinear term is represented as 
\begin{equation}
	N_h (\boldsymbol m) = \beta | { A}_h \nabla_h \boldsymbol m |^2 \boldsymbol m - \alpha \boldsymbol m \times ( \boldsymbol m \times \boldsymbol f ).
	\label{NL-discrete-1}
\end{equation}

For the sake of convenience, all coefficients are kept with four significant figures, which will not affect the convergence analysis. The actual numerical tests are kept with the original eight significant figures. In turn, the IMEX-RK3 scheme could be expressed as follows
	\begin{equation}\label{IMEX-RK3-NL}
		\left\{\begin{array}{l}
			\boldsymbol{\tilde{m}}_{1}=  \boldsymbol{{m}}_{n}, \\ {\boldsymbol{\tilde{m}}_{2}}=  \boldsymbol{{\tilde{m}}}_{1}+0.6250kN_h(\boldsymbol{{\tilde{m}}}_{1})
			+0.6250kL_h(\boldsymbol{{\tilde{m}}}_{2}), \\ {\boldsymbol{\tilde{m}}_{3}}=\boldsymbol{{\tilde{m}}}_{1}+0.1706kN_h(\boldsymbol{{\tilde{m}}}_{1})
			+0.1429kN_h(\boldsymbol{{\tilde{m}}}_{2}) -0.2359kL_h(\boldsymbol{{\tilde{m}}}_{2})+0.5494kL_h(\boldsymbol{{\tilde{m}}}_{3}),\\
			{\boldsymbol{\tilde{m}}_{4}}=  \boldsymbol{{\tilde{m}}}_{1}+0.4500kN_h(\boldsymbol{{\tilde{m}}}_{2})
			+0.5500kN_h(\boldsymbol{{\tilde{m}}}_{3})\\
			\quad \quad  \quad +0.0850kL_h(\boldsymbol{{\tilde{m}}}_{2})+0.6819kL_h(\boldsymbol{{\tilde{m}}}_{3})+0.2331kL_h(\boldsymbol{{\tilde{m}}}_{4}),\\
			\boldsymbol{{m}}_{n+1}=\boldsymbol{{\tilde{m}}}_{1}+0.0850kN_h(\boldsymbol{{\tilde{m}}}_{2})+0.6819kN_h(\boldsymbol{{\tilde{m}}}_{3})+0.2331kN_h(\boldsymbol{{\tilde{m}}}_{4})\\
			\quad  \quad \quad
			+0.0850kL_h(\boldsymbol{{\tilde{m}}}_{2})+0.6819kL_h(\boldsymbol{{\tilde{m}}}_{3})+0.2331kL_h(\boldsymbol{{\tilde{m}}}_{4}).
		\end{array}\right.
\end{equation}
Of course, this numerical system could be equivalently rewritten in the following form, to facilitate the Runge-Kutta analysis:

	\begin{small}
		\begin{align}
			&
			\frac{\tilde{\boldsymbol m}_2 - \boldsymbol m_n}{k} = 0.6250 N_h ( \tilde{\boldsymbol m}_1 ) + 0.6250 \beta \Delta_h \tilde{\boldsymbol m}_2 ,   \label{IMEX-RK3-NL-2-1}
			\\
			&
			\frac{\tilde{\boldsymbol m}_3 - \tilde{\boldsymbol m}_2}{k} = -0.4544 N_h ( \tilde{\boldsymbol m}_1 )+0.1429 N_h ( \tilde{\boldsymbol m}_2 )
			-0.8609 \beta \Delta_h  \tilde{\boldsymbol m}_2  + 0.5494 \beta \Delta_h  \tilde{\boldsymbol m}_3 ,   \label{IMEX-RK3-NL-2-2}
			\\
			&
			\frac{\tilde{\boldsymbol m}_4 - \tilde{\boldsymbol m}_3}{k} = -0.1706 N_h ( \tilde{\boldsymbol m}_1 )+0.3071 N_h ( \tilde{\boldsymbol m}_2 )+ 0.5500 N_h ( \tilde{\boldsymbol m}_3 )\nonumber
			\\
			&
			\qquad \qquad  \quad + 0.3209 \beta \Delta_h  \tilde{\boldsymbol m}_2  + 0.1325 \beta \Delta_h \tilde{\boldsymbol m}_3 + 0.2331 \beta \Delta_h  \tilde{\boldsymbol m}_4 ,   \label{IMEX-RK3-NL-2-3}
			\\
			& \frac{\boldsymbol m_{n+1} - \tilde{\boldsymbol m}_4}{k} = -0.3650 N_h ( \tilde{\boldsymbol m}_2 )+0.1319 N_h ( \tilde{\boldsymbol m}_3 )+ 0.2331 N_h ( \tilde{\boldsymbol m}_4 ) . \label{IMEX-RK3-NL-2-4}
		\end{align}
\end{small}

The main theoretical result of the convergence analysis is stated below.
\begin{theorem} \label{thm: convergence}
	Assume the exact solution $\boldsymbol \Phi$ of \eqref{equation-LL-mod} satisfies the regularity assumption: 
	Denote ${\boldsymbol m}^n$ ($n\ge0$) as the numerical solution obtained from~\eqref{IMEX-RK3-NL}, or equivalently \eqref{IMEX-RK3-NL-2-1}--\eqref{IMEX-RK3-NL-2-4}, with the initial  error satisfying $\|{P}_h \Phi (\cdot,t_0) - \boldsymbol m_0 \|_2 + (k\|\nabla_h ( {P}_{h} \Phi(\cdot,t_0) - \boldsymbol m_0 ) \|_2)^\frac12	 = {O} (h^4)$. Additionally, a linear refinement assumption that $C_1 h \le k \le C_2 h$ is made~(with $C_1,C_2$ being two positive constants) and the condition $ k \le C^{'} h$ is also made to ensure the convergence~(${C^{'}}$ is independent of $k$ and $h$, only depends on $\tilde{M}$). Then the following convergence result holds for $1 \le n\le \left\lfloor\frac{T}{k}\right\rfloor$ as $k, h$ go to zero:	
	\begin{align} \label{convergence-0}
		\| \Phi (\cdot,t_n) - \boldsymbol m^n \|_2 + (k\| \nabla_h( \Phi (\cdot,t_n) - \boldsymbol m^n) \|_2)^\frac12	&\leq {C} ( k^3+h^2) ,
	\end{align}	
	in which the constant ${C}>0$ is independent of $k$ and $h$.
\end{theorem}

First, we construct an approximate solution $\underline{\Phi} = \Phi + {{h}^{2}}{{\Phi}^{(1)}}$, so that an ${O} (h^4)$ spatial truncation error is obtained. Such a higher order consistency is important in the later analysis to bound the $\| \cdot \|_{W_h^{1,\infty}}$ norm of the numerical solution. It is noticed that $\Phi^{(1)}$ is a spatially continuous function, and its construction will be obtained using a perturbation expansion, which depends solely on the exact solution $\Phi$. Moreover, a higher $O (k^3 + h^4)$ consistency has to be satisfied with the given numerical scheme~\eqref{IMEX-RK3-NL-2-1}--\eqref{IMEX-RK3-NL-2-4}.

In more details, we introduce a higher order approximate expansion of the exact solution, since the second order spatial accuracy, associated with the centered difference approximation, is not able to control the discrete $\| \cdot \|_{W_h^{1,\infty}}$ norm of the numerical solution, which is needed in the later convergence analysis. In turn, instead of substituting the exact solution into the numerical algorithm, a careful construction of an approximate profile is performed by adding an ${O} (h^2)$ correction term to the exact solution to satisfy an ${O} (h^4)$ truncation error. Afterwards, we analyze the numerical error function between the constructed profile and the numerical solution, instead of a direct comparison between the numerical solution and exact solution. Such an improved consistency will lead to a higher order convergence estimate in the $\ell^\infty (0,T; \ell^2) \cap \ell^2 (0, T; H_h^1)$ norm, which in turn yields a desired $\| \cdot \|_{W_h^{1,\infty}}$ bound of the numerical solution, with the help of the inverse inequality. Similar techniques has been reported for a wide class of nonlinear PDEs; see the related works for the incompressible fluid equation~\cite{E95, E02, STWW2003, STWW2007, Wang2000, Wang2002, Wang2004}, various gradient equations~\cite{baskaran13b, guan17a, guan14a, LiX21a, LiX23a, LiX24a}, the porous medium equation based on the energetic variational approach~\cite{duan20a, duan22a, duan22b}, nonlinear wave equation~\cite{WangL15}, etc.

An application of the centered finite difference discretization to the exact solution $\Phi$ gives
\begin{equation} \label{consistency-2-1}
	\Phi_{t}= \beta | { A}_h \nabla_h \Phi |^2 \Phi - \alpha \Phi \times ( \Phi \times \boldsymbol f )
	+ \beta \Delta_h \Phi + h^2 \boldsymbol g^{(2)} +{O} (h^4)   ,
\end{equation}
which comes from the Taylor expansion in space. In more details, the function $\boldsymbol g^{(2)}$ is smooth enough and only depends on the higher order derivatives of $\Phi$. Subsequently, the spatial correction function $\Phi^{(1)}$ is given by the solution of the following linear differential equation
\begin{equation}
	\begin{aligned}
		\partial_t \Phi^{(1)} = & \beta \Delta \Phi^{(1)}
		+ \beta \Big( |  \nabla \Phi |^2 \Phi^{(1)}
		+  2 ( \nabla \Phi \cdot \nabla \Phi^{(1)} ) \Phi \Big)
		\\
		&
		- \alpha \Big( \Phi \times \Phi^{(1)} + \Phi^{(1)} \times ( \Phi + \boldsymbol f ) \Big)
		- \boldsymbol g^{(2)}  ,  \quad
		\Phi^{(1)} ( \cdot , \, t=0) \equiv 0
	\end{aligned}
	\label{consistency-2-2}
\end{equation}
with homogeneous Neumann boundary condition. In fact, \eqref{consistency-2-2} is a linear parabolic PDE, and the existence and uniqueness of its solution could be derived by making use of a standard Galerkin procedure and Sobolev estimates, following the classical techniques for time-dependent parabolic equation~\cite{temam01}.	Moreover, the solution of~\eqref{consistency-2-2} depends only on the exact profile $\Phi$ and is smooth enough. Similar to~\eqref{consistency-2-1}, an application of the finite difference discretization to $\Phi^{(1)}$ gives
\begin{equation}
	\begin{aligned}
		\partial_t \Phi^{(1)} = &
		\beta \Big( | { A}_h \nabla_h \Phi |^2 \Phi^{(1)}
		+  2 ( { A}_h \nabla_h \Phi \cdot { A}_h \nabla_h \Phi^{(1)} ) \Phi \Big)
		\\
		&
		- \alpha \Big( \Phi \times \Phi^{(1)} + \Phi^{(1)} \times ( \Phi + \boldsymbol f ) \Big)
		+ \beta \Delta_h \Phi^{(1)}  - \boldsymbol g^{(2)}
		+ {O} (h^2) .
	\end{aligned}
	\label{consistency-2-3}
\end{equation}
In turn, a combination of~\eqref{consistency-2-1} and \eqref{consistency-2-3} leads to the following higher order consistency estimate for $\underline{\Phi} = \Phi + h^2 \Phi^{(1)}$:
\begin{equation} \label{consistency-2-4}
	\underline{ \Phi}_{t}= \beta | { A}_h \nabla_h \underline{\Phi} |^2 \underline{\Phi}
	- \alpha \underline{\Phi} \times ( \underline{\Phi} \times \boldsymbol f )
	+ \beta \Delta_h \underline{\Phi} + {O} (h^4)   .
\end{equation}

Moreover, we extend the approximate profile $\underline{\Phi}$ to the numerical ghost points, according to the extrapolation formula:
\begin{align}\label{extrapolation}
	\underline{\Phi}_{i,j,0}=\underline{\Phi}_{i,j,1}, \quad \underline{\Phi}_{i, j, N_{z+1}}=\underline{\Phi}_{i,j,N_z},
\end{align}
and the extrapolation for other boundaries can be formulated in the same manner. In addition, we are able to prove that such an extrapolation yields a higher order ${O}\left(h^5\right)$ approximation, due to the fact that
\begin{equation*}
	\partial_z^3 \Phi = 0 , \quad \partial_z \Phi^{(1)} = 0 ,  \quad
	\mbox{at} \quad z = 0, 1 .
\end{equation*}

Given the exact solution $\Phi$, we denote $\underline{\Phi}^n = \underline{\Phi} (\cdot , t^n)$. In addition, another three intermediate approximate solutions need to be constructed at each time step, to facilitate the Runge-Kutta analysis, following the same algorithm as in~\eqref{IMEX-RK3-NL}:

	\begin{align}
		&  \tilde{\Phi}^{n,(2)} = \underline{\Phi}^n + 0.6250k N_h (\underline{\Phi}^{n}) + 0.6250k\beta \Delta_h \tilde{\Phi}^{n,(2)} ,   \label{consistency-1-1}
		\\
		&  \tilde{\Phi}^{n,(3)} = \underline{\Phi}^n +0.1706k N_h ( \underline{\Phi}^{n})  +0.1429k N_h ( \tilde{\Phi}^{n,(2)} )\nonumber
		\\
		&
		\qquad \quad-0.2359k \beta \Delta_h \tilde{\Phi}^{n,(2)} + 0.5494k \beta \Delta_h \tilde{\Phi}^{n,(3)}  ,   \label{consistency-1-2}
		\\
		&  \tilde{\Phi}^{n,(4)} = \underline{\Phi}^n +0.4500k N_h ( \tilde{\Phi}^{n,(2)} ) +0.5500k   N_h ( \tilde{\Phi}^{n,(3)} ) \nonumber
		\\
		&
		\qquad \quad+0.0850k \beta \Delta_h  \tilde{\Phi}^{n,(2)} +0.6819k \beta \Delta_h  \tilde{\Phi}^{n,(3)} +0.2331 k\beta \Delta_h  \tilde{\Phi}^{n,(4)} ,   \label{consistency-1-3}
\end{align}
in which the homogeneous discrete Neumann boundary condition (similar to~\eqref{extrapolation}) is imposed for $\tilde{\Phi}^{n,(j)}$, $j = 2, 3, 4$. Moreover, the careful Taylor expansion (related to the IMEX-RK3 method) reveals the following consistency estimate of the constructed solution at the next time step with $\| \tau^{n} \|_2 \le { C} (k^3 + h^4)$.

	\begin{align}
		\underline{\Phi}^{n+1} =~&\underline{\Phi}^n + 0.0850k N_h ( \tilde{\Phi}^{n,(2)} ) + 0.6819k N_h ( \tilde{\Phi}^{n,(3)} )  +0.2331 k N_h ( \tilde{\Phi}^{n,(4)} ) \nonumber
		\\
		&
		+0.0850  k\beta \Delta_h  \tilde{\Phi}^{n,(2)}   + 0.6819  k\beta \Delta_h  \tilde{\Phi}^{n,(3)}
		+0.2331   k\beta \Delta_h  \tilde{\Phi}^{n,(4)}  + k\tau^{n}. \label{consistency-2}
\end{align}

Clearly, by  performing a similar transformation as in~\eqref{IMEX-RK3-NL-2-1}--\eqref{IMEX-RK3-NL-2-4}, the constructed profiles $\underline{\Phi}^n$, $\underline{\Phi}^{n+1}$ and $\tilde{\Phi}^{n,(j)}$ ($j = 2, 3, 4$) satisfy the numerical system:

	\begin{small}
		\begin{align}
			&\frac{\tilde{\Phi}^{n,(2)} - \underline{\Phi}^n}{k} =0.6250 N_h ( \underline{\Phi}^{n})  + 0.6250 \beta \Delta_h \tilde{\Phi}^{n,(2)} ,   \label{consistency-3-1}
			\\
			& \frac{\tilde{\Phi}^{n,(3)} - \tilde{\Phi}^{n,(2)}}{k}  = -0.4544 N_h ( \underline{\Phi}^{n})  + 0.1429 N_h ( \tilde{\Phi}^{n,(2)} ) -0.8609 \beta \Delta_h \tilde{\Phi}^{n,(2)} + 0.5494 \beta \Delta_h \tilde{\Phi}^{n,(3)}  ,   \label{consistency-3-2}
			\\
			&\frac{\tilde{\Phi}^{n,(4)} - \tilde{\Phi}^{n,(3)}}{k}  = -0.1706 N_h ( \underline{\Phi}^{n} ) +0.3071 N_h ( \tilde{\Phi}^{n,(2)} ) + 0.5500 N_h ( \tilde{\Phi}^{n,(3)} )  \nonumber
			\\
			&\qquad   \qquad \qquad +0.3209 \beta \Delta_h \tilde{\Phi}^{n,(2)} + 0.1325 \beta \Delta_h \tilde{\Phi}^{n,(3)} + 0.2331 \beta \Delta_h \tilde{\Phi}^{n,(4)} ,   \label{consistency-3-3}
			\\
			&\frac{\underline{\Phi}^{n+1} - \tilde{\Phi}^{n,(4)}}{k} = -0.3650 N_h (\tilde{\Phi}^{n,(2)}) +0.1319 N_h (\tilde{\Phi}^{n,(3)}) +0.2331 N_h (\tilde{\Phi}^{n,(4)})+\tau ^n . \label{consistency-3-4}
		\end{align}
\end{small}

Since the constructed profiles $\tilde{\Phi}^{n,(j)}$ ($j = 2, 3, 4$) only rely on the approximate solution $\underline{\Phi}^n$, the consistency estimate reveals that
\begin{equation}
	\begin{aligned}
		\|\underline{\Phi}^{n} \|_\infty, \| \tilde{\Phi}^{n,(j)} \|_\infty &\le \frac98 , \\
		\| \nabla_h \underline{\Phi}^{n} \|_\infty, \| \nabla_h \tilde{\Phi}^{n,(j)} \|_\infty
		&\le { C}^* ,
		\quad  j = 2, 3, 4 .
	\end{aligned}
	\label{consistency-4}
\end{equation}
Therefore, we define the numerical error functions as follows, at a point-wise level:
\begin{equation}
	\begin{aligned}
		&
		\underline{\boldsymbol e}^k = \underline{\Phi}^k - \boldsymbol m_k ,  \, \, \, k=n, n+1 ,
		\\
		&
		\tilde{\boldsymbol e}^{n,(j)} = \tilde{\boldsymbol Phi}^{n,(j)} - \tilde{\boldsymbol m}_j , \, \, \, j =2 , 3, 4 .
	\end{aligned}
	\label{error function-1}
\end{equation}
Again, instead of a direct comparison between the numerical and exact solutions, we analyze the error between the numerical solution and the constructed approximate solution, due to its higher order consistency. Moreover, the following nonlinear error terms are introduced to simplify the notation: 
\begin{equation}
	\begin{aligned}
		{ NLE}^{n,(1)} &= N_h ( \underline{\Phi}^{n} )  -  N_h ( \boldsymbol m_n ), \\
		{ NLE}^{n,(j)} &= N_h ( \tilde{\Phi}^{n,(j)} )  -  N_h ( \tilde{\boldsymbol m}_j ) ,  \quad  j =2 , 3, 4 .
	\end{aligned}
	\label{NL error-def}
\end{equation}
In turn, a subtraction of the numerical algorithm~\eqref{IMEX-RK3-NL-2-1}--\eqref{IMEX-RK3-NL-2-4} from the consistency estimate~\eqref{consistency-3-1}--\eqref{consistency-3-4} leads to the following numerical error evolution system:

	\begin{small}
		\begin{align}
			&\frac{\tilde{\boldsymbol e}^{n,(2)} - \underline{\boldsymbol e}^n}{k} =0.6250 { NLE}^{n,(1)}  + 0.6250 \beta \Delta_h \tilde{\boldsymbol e}^{n,(2)} ,   \label{consistency-5-1}
			\\
			& \frac{\tilde{\boldsymbol e}^{n,(3)} - \tilde{\boldsymbol e}^{n,(2)}}{k}  = -0.4544 { NLE}^{n,(1)}  + 0.1429 { NLE}^{n,(2)} -0.8609 \beta \Delta_h \tilde{\boldsymbol e}^{n,(2)} + 0.5494 \beta \Delta_h \tilde{\boldsymbol e}^{n,(3)}  ,   \label{consistency-5-2}
			\\
			&\frac{\tilde{\boldsymbol e}^{n,(4)} - \tilde{\boldsymbol e}^{n,(3)}}{k}  = -0.1706 { NLE}^{n,(1)} +0.3071 { NLE}^{n,(2)}+ 0.5500{ NLE}^{n,(3)}  \nonumber
			\\
			&\qquad   \qquad \qquad +0.3209 \beta \Delta_h \tilde{\boldsymbol e}^{n,(2)} + 0.1325 \beta \Delta_h \tilde{\boldsymbol e}^{n,(3)} + 0.2331 \beta \Delta_h \tilde{\boldsymbol e}^{n,(4)} ,   \label{consistency-5-3}
			\\
			&\frac{\underline{\boldsymbol e}^{n+1} - \tilde{\boldsymbol e}^{n,(4)}}{k} = -0.3650 { NLE}^{n,(2)} +0.1319 { NLE}^{n,(3)} +0.2331 { NLE}^{n,(4)}+\tau ^n . \label{consistency-5-4}
		\end{align}
\end{small}

In order to established the convergence analysis, it is necessary to bound the nonlinear error term. For the sake of notation simplicity, a uniform constant ${C}$ is used to represent all controllable constants.
\begin{lemma}  \label{lem: NL error}
	Under the regularity estimate~\eqref{consistency-4} for the constructed profiles, and the following bound in the IMEX-RK stages
	\begin{equation}
		\| \tilde{\boldsymbol m}_j \|_\infty  \le \frac54, \quad
		\| \nabla_h \tilde{\boldsymbol m}_j \|_\infty \le \tilde{ C} := { C}^* +1 ,
		\quad  j = 2, 3, 4 ,
		\label{bound-m-1}
	\end{equation}
	an $\| \cdot \|_2$ estimate for the nonlinear error terms is available:
	\begin{align}
		&\| { NLE}^{n, (1)} \|_2  \le \tilde{M} ( \|  \underline{\boldsymbol e}^{n}  \|_2
		+ \| \nabla_h  \underline{\boldsymbol e}^{n}  \|_2 ) , \label{NLE-error-1}
		\\
		&
		\| { NLE}^{n, (j)} \|_2  \le \tilde{M} ( \|  \tilde{\boldsymbol e}^{n,(j)}  \|_2
		+ \| \nabla_h  \tilde{\boldsymbol e}^{n,(j)}  \|_2 ) ,   \quad  j = 2, 3, 4,
		\label{NLE-error-2}
	\end{align}
	in which $\tilde{M}$ only depends on $\alpha$, $\beta$, ${ C}^*$, $\tilde{ C}$, and the external force term $\boldsymbol f$.
\end{lemma}

\begin{proof}
	For simplicity, only the nonlinear error term $\| { NLE}^{n, (1)} \|_2$ is considered, and the estimate of $\| { NLE}^{n, (j)} \|_2$ could be derived in the same manner. In fact, a careful expansion of the term $ { NLE}^{n, (1)} $ indicates that
	\begin{equation}
		\begin{aligned}
			&{ NLE}^{n, (1)} = N_h (\underline{\Phi}^{n} ) - N_h ( \boldsymbol m_n  )
			\\
			&= \beta | { A}_h \nabla_h \underline{\Phi}^{n} |^2
			\underline{\boldsymbol e}^{n} + \beta \Big(
			{ A}_h \nabla_h  ( \underline{\Phi}^{n} + \boldsymbol m_n )
			\cdot { A}_h \nabla_h  \underline{\boldsymbol e}^{n}  \Big) \boldsymbol m_n
			\\
			&
			\quad - \alpha \boldsymbol m_n \times ( \underline{\boldsymbol e}^{n} \times \boldsymbol f )
			- \alpha \underline{\boldsymbol e}^{n} \times ( \underline{\Phi}^{n} \times \boldsymbol f ).
		\end{aligned}
		\label{NL error-2}
	\end{equation}
	As a result, a direct application of discrete H\"older inequality yields 
	\begin{align}
		&
		\Big\| \beta | { A}_h \nabla_h \underline{\Phi}^{n} |^2  \underline{\boldsymbol e}^{n} \Big\|_2
		\le \beta  \|   \nabla_h \underline{\Phi}^{n} \|_\infty^2 \cdot  \| \underline{\boldsymbol e}^{n} \|_2
		\le C \beta ( { C}^* )^2  \| \underline{\boldsymbol e}^{n} \|_2 ,
		\label{NL error-4-1}
		\\
		&
		\Big\| \beta \Big(
		{ A}_h \nabla_h  ( \underline{\Phi}^{n} + \boldsymbol m_n )
		\cdot { A}_h \nabla_h  \underline{\boldsymbol e}^{n}  \Big) \boldsymbol m_n  \Big\|_2    \nonumber
		\\
		& \le
		\beta (  \| \nabla_h  \underline{\Phi}^{n} \|_\infty + \| \nabla_h \boldsymbol m_n \|_\infty )
		\cdot \| \nabla_h  \underline{\boldsymbol e}^{n} \|_2 \cdot \| \boldsymbol m_n \|_\infty  \nonumber
		\\
		& \le
		C \beta (  { C}^* + \tilde{ C} )  \| \nabla_h  \underline{\boldsymbol e}^{n} \|_2 ,
		\label{NL error-4-2}
		\\
		&
		\Big\|  \alpha \boldsymbol m_n \times ( \underline{\boldsymbol e}^{n} \times \boldsymbol f )   \Big\|_2
		\le \alpha  \| \boldsymbol m_n \|_\infty \cdot \| \boldsymbol f \|_\infty \cdot \| \underline{\boldsymbol e}^{n} \|_2  \nonumber
		\\
		&
		\le \frac{5 \alpha}{4} C_0 \|\underline{\boldsymbol e}^{n} \|_2
		\le C \alpha C_0  \| \underline{\boldsymbol e}^{n} \|_2 ,  \label{NL error-4-3}
		\\
		&\Big\|  \alpha \underline{\boldsymbol e}^{n} \times ( \underline{\Phi}^{n} \times \boldsymbol f )  \Big\|_2
		\le \alpha  \| \underline{\Phi}^{n} \|_\infty \cdot \| \boldsymbol f \|_\infty \cdot \| \underline{\boldsymbol e}^{n} \|_2  \nonumber
		\\
		&
		\le \frac{9 \alpha}{8} C_0 \| \underline{\boldsymbol e}^{n} \|_2
		\le C \alpha C_0  \| \underline{\boldsymbol e}^{n} \|_2 ,  \label{NL error-4-4}
	\end{align}
	in which the summation-by-parts formula~\eqref{summation} and the bound~\eqref{bound-m-1} have been applied, along with the fact that $\| \cdot \|_\infty$ bound for the external force term: $\| \boldsymbol f \|_\infty \le C_0$. A substitution of~\eqref{NL error-4-1}--\eqref{NL error-4-4} into \eqref{NL error-2} leads to the nonlinear error estimate \eqref{NLE-error-1}, by taking $\tilde{M} =  C (  \beta ( (  { C}^* )^2 +  { C}^* + \tilde{ C})+ \alpha C_0  ) $. The proof of Lemma~\ref{lem: NL error} is finished.
\end{proof}

Before proceeding into the formal error estimate, the following a-priori assumption is made for the error function at the previous time step:
\begin{equation} \label{bound-2}
	\| \underline{\boldsymbol e}^n \|_2 \le k^{\frac{11}{4}} + h^{\frac{15}{4}}, \qquad \| \nabla_h \underline{\boldsymbol e}^n \|_2 \le k^{\frac{9}{4}} + h^{\frac{13}{4}} .
\end{equation}
As stated above, the multi-stage nature of the third order Runge-Kutta scheme, as well as the complicated nonlinear terms, make the theoretical analysis highly challenging. Therefore, the error estimates at each RK stage are separately discussed. 

\noindent
{\bf Error estimate at Runge-Kutta Stage 1}  \, \, In the first stage, by taking a discrete inner product with~\eqref{consistency-5-1} by $2 \tilde{\boldsymbol e}^{n, (2)}$, it follow 
	\begin{equation}
		\| \tilde{\boldsymbol e}^{n, (2)} \|_2^2 - \| \underline{\boldsymbol e}^n \|_2^2 + \| \tilde{\boldsymbol e}^{n, (2)} - \underline{\boldsymbol e}^n \|_2^2
		+ \frac{5}{4}\beta k \| \nabla_h \tilde{\boldsymbol e}^{n, (2)} \|_2^2 = \frac{5}{4} k \langle { NLE}^{n, (1)} ,  \tilde{\boldsymbol e}^{n, (2)} \rangle ,
		\label{convergence-1}
\end{equation}
based on an application of the summation-by-parts formula~\eqref{sum1}, and the discrete homogeneous Neumann boundary condition for $\tilde{\boldsymbol e}^{n, (2)}$. In terms of the inner product term associated with the nonlinear error, the following estimates are derived: 
\begin{equation}
	\begin{aligned}
		&
		\langle { NLE}^{n, (1)} ,  \tilde{\boldsymbol e}^{n, (2)} \rangle
		\le   \| { NLE}^{n, (1)} \|_2 \cdot \| \tilde{\boldsymbol e}^{n, (2)}  \|_2
		\\
		&
		\le \tilde{M} ( \|  \underline{\boldsymbol e}^{n}  \|_2
		+ \| \nabla_h  \underline{\boldsymbol e}^{n}  \|_2 )  \cdot \| \tilde{\boldsymbol e}^{n, (2)}  \|_2
		\\
		&
	\le \frac{\tilde{M}}{2} \|  \underline{\boldsymbol e}^{n}  \|_2^2
			+ \frac{\beta}{2500} \| \nabla_h  \underline{\boldsymbol e}^{n}  \|_2^2
			+ (  \frac{\tilde{M}}{2} + \frac{625 \tilde{M}^2}{\beta} ) \| \tilde{\boldsymbol e}^{n, (2)}  \|_2^2, 
	\end{aligned}
	\label{convergence-1-1}
\end{equation}
in which the Young's inequality has been applied in the last step. In turn, we denote ${C}_1=\frac{5}{4}(  \frac{\tilde{M}}{2} + \frac{625 \tilde{M}^2}{\beta} ) $, and see that the right-hand side of \eqref{convergence-1} is bounded as 
	\begin{align}
		&
		\frac{5k}{4} \langle { NLE}^{n, (1)} ,  \tilde{\boldsymbol e}^{n, (2)} \rangle \le \frac{5}{8}\tilde{M}k \|  \underline{\boldsymbol e}^{n}  \|_2^2 + \frac{\beta k}{2000} \| \nabla_h  \underline{\boldsymbol e}^{n}  \|_2^2
		+ {C}_1 k  \| \tilde{\boldsymbol e}^{n, (2)} \|_2^2 .   	
		\label{convergence-1-3}
\end{align}
Its substitution into \eqref{convergence-1} yields
	\begin{align}
		&
		\| \tilde{\boldsymbol e}^{n, (2)} \|_2^2 - \| \underline{\boldsymbol e}^n \|_2^2 + \| \tilde{\boldsymbol e}^{n, (2)} - \underline{\boldsymbol e}^n \|_2^2
		+ \frac54 \beta k \| \nabla_h \tilde{\boldsymbol e}^{n, (2)} \|_2^2 \nonumber
		\\
		&
		\le \frac{5}{8}\tilde{M}k \|  \underline{\boldsymbol e}^{n}  \|_2^2
		+ \frac{\beta k}{2000}\| \nabla_h  \underline{\boldsymbol e}^{n}  \|_2^2+ {C}_1k  \| \tilde{\boldsymbol e}^{n, (2)} \|_2^2.
		\label{convergence-1-4}
\end{align} 
Furthermore, under the linear refinement requirement $C_1 h \le k \le C_2 h$ and the assumption that $k$ is sufficiently small, the rough error estimates could be obtained as follows 
\begin{align}
	&\| \tilde{\boldsymbol e}^{n, (2)} \|_2 \le  \Big( \frac{1+ {C}_2 k}{1 - {C}_1 k } \Big)^\frac12 (k^{\frac{9}{4}} + h^{\frac{5}{4}}) \le 2 (k^{\frac{9}{4}} + h^{\frac{5}{4}} ),
	\label{bound-stage-1-1}
	\\
	&\| \nabla_h \tilde{\boldsymbol e}^{n, (2)} \|_2 \le   \beta^{-\frac12} k^{-\frac12}
	\Big( 1+ {C}_2 k \Big)^\frac12  (k^{\frac{9}{4}} + h^{\frac{5}{4}})
	\le k^{\frac{7}{4}} + h^{\frac{3}{4}} , 
	\label{bound-stage-1-2}
\end{align} 
by taking ${C}_2=\frac{5\tilde{M}}{8}+\frac{\beta }{2000}$ due to the a-priori assumption~\eqref{bound-2}. Likewise, the $\| \cdot \|_\infty$ bound for both the numerical error function $\tilde{\boldsymbol e}^{n, (2)}$ and the numerical solution $\tilde{\boldsymbol m}_2$ are available: 
\begin{align}
	&
	\| \tilde{\boldsymbol e}^{n, (2)}  \|_\infty \le \gamma {h}^{-1/2 }
	( \| \tilde{\boldsymbol e}^{n, (2)}  \|_2 + \| \nabla_h \tilde{\boldsymbol e}^{n, (2)}  \|_2 )
	\le  \gamma \Big( \frac{k^{\frac{7}{4}}}{h^\frac12} + h^{\frac{9}{4}} \Big) \le \frac18 ,  	
	\label{bound-stage-1-3} 	
	\\
	& 	
	\|  \nabla_h \tilde{\boldsymbol e}^{n, (2)} \|_\infty \le \gamma {h}^{-\frac{3}{2}} \| \nabla_h \tilde{\boldsymbol e}^{n, (2)} \|_2
	\le  \gamma \Big( \frac{k^{\frac{7}{4}}}{h^\frac32} + h^{\frac54} \Big) \le 1 , 	
	\label{bound-stage-1-4}
	\\
	&
	\| \tilde{\boldsymbol m}_2 \|_\infty \le \| \tilde{\Phi}^{n, (2)} \|_\infty
	+ \| \tilde{\boldsymbol e}^{n, (2)}  \|_{\infty} \le \frac98 + \frac18 = \frac54 ,  	
	\label{bound-stage-1-5} 	
	\\
	& 	
	\| \nabla_h \tilde{\boldsymbol m}_2 \|_\infty \le \| \nabla_h \tilde{\Phi}^{n, (2)} \|_\infty
	+ \|  \nabla_h \tilde{\boldsymbol e}^{n, (2)} \|_\infty \le { C}^* + 1 = \tilde{ C} . 	\label{bound-stage-1-6} 	
\end{align}

\noindent
{\bf Error estimate at Runge-Kutta Stage 2}  \, \, Similarly, taking a discrete inner product with~\eqref{consistency-5-2} by $2 \tilde{\boldsymbol e}^{n, (3)}$ leads to
	\begin{small}
		\begin{equation}
			\begin{aligned}
				&
				\| \tilde{\boldsymbol e}^{n, (3)} \|_2^2 - \| \tilde{\boldsymbol e}^{n, (2)} \|_2^2 + \| \tilde{\boldsymbol e}^{n, (3)} - \tilde{\boldsymbol e}^{n, (2)} \|_2^2
				+ 1.0988 \beta k \| \nabla_h \tilde{\boldsymbol e}^{n, (3)} \|_2^2=
				\\
				&- 0.9088k\langle { NLE}^{n, (1)}, \tilde{\boldsymbol e}^{n, (3)} \rangle +  0.2858k\langle { NLE}^{n, (2)}, \tilde{\boldsymbol e}^{n, (3)} \rangle + 1.7218 \beta k \langle \nabla_h \tilde{\boldsymbol e}^{n, (2)},  \nabla_h \tilde{\boldsymbol e}^{n, (3)} \rangle.
			\end{aligned}
			\label{convergence-2-1}
		\end{equation}
\end{small}
A bound for the last term on the right-hand side is straightforward 
\begin{equation}
	\begin{aligned}
		&
		\langle \nabla_h \tilde{\boldsymbol e}^{n, (2)} ,  \nabla_h \tilde{\boldsymbol e}^{n, (3)} \rangle
		\le  \frac12 ( \| \nabla_h \tilde{\boldsymbol e}^{n, (2)} \|_2^2 + \|  \nabla_h \tilde{\boldsymbol e}^{n, (3)} \|_2^2 ) ,\quad \mbox{so that}
		\\
		&
		 1.7218 \beta k \langle \nabla_h \tilde{\boldsymbol e}^{n, (2)} ,  \nabla_h \tilde{\boldsymbol e}^{n, (3)} \rangle
			\le 0.8609 \beta k ( \| \nabla_h \tilde{\boldsymbol e}^{n, (2)} \|_2^2 + \|  \nabla_h \tilde{\boldsymbol e}^{n, (3)} \|_2^2 ).
	\end{aligned}
	\label{convergence-2-2}
\end{equation}
The nonlinear error terms, as well as the corresponding inner product, could be analyzed in a similar manner with the help of Lemma~\ref{lem: NL error}, which implies the estimates as follows:
\begin{align}
	\langle { NLE}^{n, (1)}, \tilde{\boldsymbol e}^{n, (3)} \rangle \le
	\frac{\tilde{M}}{2} \|  \underline{\boldsymbol e}^{n}  \|_2^2 + \frac{5\beta}{4544} \| \nabla_h  \underline{\boldsymbol e}^{n}  \|_2^2
		+ (\frac{\tilde{M}}{2}+\frac{1136 \tilde{M}^2}{5\beta}) \| \tilde{\boldsymbol e}^{n, (3)} \|_2^2 ,   \label{convergence-2-3}
	\\
	\langle { NLE}^{n, (2)}, \tilde{\boldsymbol e}^{n, (3)} \rangle \le
	\frac{\tilde{M}}{2} \|  \tilde{\boldsymbol e}^{n, (2)} \|_2^2
		+ \frac{5\beta}{1429} \| \nabla_h  \tilde{\boldsymbol e}^{n, (2)}  \|_2^2
		+ (\frac{\tilde{M}}{2}+\frac{1429 \tilde{M}^2}{20\beta}) \| \tilde{\boldsymbol e}^{n, (3)} \|_2^2 , \label{convergence-2-4}
\end{align}
under the regularity estimate~\eqref{consistency-4} and the bound~\eqref{NLE-error-1}--\eqref{NLE-error-2}. Subsequently, a substitution of~\eqref{convergence-2-2}--\eqref{convergence-2-4} into \eqref{convergence-2-1} yields
	\begin{equation}
		\begin{aligned}
			&
			\| \tilde{\boldsymbol e}^{n, (3)} \|_2^2 - \| \tilde{\boldsymbol e}^{n, (2)} \|_2^2 + \| \tilde{\boldsymbol e}^{n, (3)} - \tilde{\boldsymbol e}^{n, (2)} \|_2^2
			+ 0.2379 \beta k \| \nabla_h \tilde{\boldsymbol e}^{n, (3)} \|_2^2 - 0.8619 \beta k
			\\
			&\| \nabla_h \tilde{\boldsymbol e}^{n, (2)} \|_2^2 \le  0.4544\tilde{M}k\| \underline{\boldsymbol e}^{n}  \|_2^2
			+ \frac{\beta k}{1000}   \| \nabla_h \underline{\boldsymbol e}^{n}  \|_2^2 + 0.1429\tilde{M}k \| \tilde{\boldsymbol e}^{n, (2)}  \|_2^2 + {C}_3 k \| \tilde{\boldsymbol e}^{n, (3)} \|_2^2,
		\end{aligned}
		\label{convergence-2-5}
\end{equation}
with 	 ${C}_3=0.9088(\frac{\tilde{M}}{2}+\frac{1136 \tilde{M}^2}{5\beta})+0.2858(\frac{\tilde{M}}{2}+\frac{1429 \tilde{M}^2}{20\beta})$. Furthermore, its combination with~\eqref{convergence-1-4} indicates that	
	\begin{equation}
		\begin{aligned}
			&
			\| \tilde{\boldsymbol e}^{n, (3)} \|_2^2 - \| \underline{\boldsymbol e}^n \|_2^2 + \| \tilde{\boldsymbol e}^{n, (3)} - \tilde{\boldsymbol e}^{n, (2)} \|_2^2 + \| \tilde{\boldsymbol e}^{n, (2)} - \underline{\boldsymbol e}^n \|_2^2
			\\
			&+ 0.3881 \beta k\| \nabla_h \tilde{\boldsymbol e}^{n, (2)} \|_2^2
			+ 0.2379 \beta k \| \nabla_h \tilde{\boldsymbol e}^{n, (3)} \|_2^2 \le  (0.4544+\frac58)\tilde{M}k\|\underline{\boldsymbol e}^n \|_2^2 
			\\
			&  +\frac{3\beta k}{2000}\| \nabla_h \underline{\boldsymbol e}^{n}  \|_2^2+ (0.1429\tilde{M}+{C}_1 )k \|  \tilde{\boldsymbol e}^{n,(2)}  \|_2^2
			+  {C}_3  k \| \tilde{\boldsymbol e}^{n, (3)} \|_2^2 .
		\end{aligned}
		\label{convergence-2-6}
\end{equation}	
Applying the a-priori estimate~\eqref{bound-2} and ~\eqref{bound-stage-1-1}, it follows that
\begin{align}
	\| \tilde{\boldsymbol e}^{n, (3)} \|_2 &\le  \Big( \frac{1+ {C}_4 k}{1 - {C}_3  k } \Big)^\frac12 (k^{\frac{9}{4}} + h^{\frac{13}{4}}) \le 2 (k^{\frac{9}{4}} + h^{\frac{13}{4}}),
	\label{bound-stage-2-1}
	\\
	\| \nabla_h \tilde{\boldsymbol e}^{n, (3)} \|_2 &\le 	\sqrt{5} \beta^{-\frac12} k^{-\frac12}
	\Big( 1+ {C}_4 k \Big)^\frac12  (k^{\frac{9}{4}} + h^{\frac{13}{4}})
	\le  k^{\frac{7}{4}} + h^{\frac{11}{4}}  ,
	\label{bound-stage-2-2}
\end{align}
by taking~${C}_4 ={C}_1 +1.2223\tilde{M}+\frac{3\beta }{2000}$ and under the linear refinement requirement $C_1 h \le k \le C_2 h$. In turn, the $\| \cdot \|_\infty$ bound for both the numerical error function $\tilde{\boldsymbol e}^{n, (3)}$ and the numerical solution $\tilde{\boldsymbol m}_3$ are revealed:
\begin{align}
	&
	\| \tilde{\boldsymbol e}^{n, (3)}  \|_{\infty} \le \gamma {h}^{-1/2 }
	( \| \tilde{\boldsymbol e}^{n, (3)}  \|_2 + \| \nabla_h \tilde{\boldsymbol e}^{n, (3)}  \|_2 )
	\le  \gamma \Big( \frac{k^{\frac{7}{4}}}{h^\frac12} + h^{\frac94} \Big) \le \frac18 ,  	
	\label{bound-stage-2-3} 	
	\\
	& 	
	\|  \nabla_h \tilde{\boldsymbol e}^{n, (3)} \|_\infty \le \gamma {h}^{-\frac{3}{2}} \| \nabla_h \tilde{\boldsymbol e}^{n, (3)} \|_2 \le 1 , 	
	\label{bound-stage-2-4} 	
	\\
	&
	\| \tilde{\boldsymbol m}_3 \|_\infty \le \| \tilde{\Phi}^{n, (3)} \|_\infty
	+ \| \tilde{\boldsymbol e}^{n, (3)}  \|_{\infty} \le \frac98 + \frac18 = \frac54 ,  	
	\label{bound-stage-2-5} 	
	\\
	& 	
	\| \nabla_h \tilde{\boldsymbol m}_3 \|_\infty \le \| \nabla_h \tilde{\Phi}^{n, (3)} \|_\infty
	+ \|  \nabla_h \tilde{\boldsymbol e}^{n, (3)} \|_\infty \le { C}^* + 1 = \tilde{ C} . 	
	\label{bound-stage-2-6} 	
\end{align}

\noindent
{\bf Error estimate at Runge-Kutta Stage 3}  \, \, Taking a discrete inner product with~\eqref{consistency-5-3} by $2 \tilde{\boldsymbol e}^{n, (4)}$ yields 
	\begin{equation}
		\begin{aligned}
			&
			\| \tilde{\boldsymbol e}^{n, (4)} \|_2^2 - \| \tilde{\boldsymbol e}^{n, (3)} \|_2^2 + \| \tilde{\boldsymbol e}^{n, (4)} - \tilde{\boldsymbol e}^{n, (3)} \|_2^2+ 0.4662 \beta k\| \nabla_h \tilde{\boldsymbol e}^{n, (4)} \|_2^2 
			\\
			&
			= - 0.3412k \langle { NLE}^{n, (1)} ,  \tilde{\boldsymbol e}^{n, (4)} \rangle
			+0.6142 k \langle { NLE}^{n, (2)} ,  \tilde{\boldsymbol e}^{n, (4)} \rangle+ 1.1k \langle { NLE}^{n, (3)} ,  \tilde{\boldsymbol e}^{n, (4)} \rangle
			\\
			&-0.6418 \beta k \langle \nabla_h \tilde{\boldsymbol e}^{n, (2)} ,  \nabla_h \tilde{\boldsymbol e}^{n, (4)} \rangle-0.265 \beta k \langle \nabla_h \tilde{\boldsymbol e}^{n, (3)} ,  \nabla_h \tilde{\boldsymbol e}^{n, (4)} \rangle.
		\end{aligned}
		\label{convergence-3-1}
\end{equation}
As described above, the nonlinear inner product for gradient terms on the right-hand side could be controlled in the same way as in~\eqref{convergence-2-2}: 
\begin{equation}
	\begin{aligned}
		&
		\langle \nabla_h \tilde{\boldsymbol e}^{n, (2)} ,  \nabla_h \tilde{\boldsymbol e}^{n, (4)} \rangle
		\le  \frac12 ( \| \nabla_h \tilde{\boldsymbol e}^{n, (2)} \|_2^2 + \|  \nabla_h \tilde{\boldsymbol e}^{n, (4)} \|_2^2 ) ,\quad \mbox{so that}
		\\
		&
		0.6418 \beta k \langle \nabla_h \tilde{\boldsymbol e}^{n, (2)} ,  \nabla_h \tilde{\boldsymbol e}^{n, (4)} \rangle
			\le  0.3209 \beta k ( \| \nabla_h \tilde{\boldsymbol e}^{n, (2)} \|_2^2 + \|  \nabla_h \tilde{\boldsymbol e}^{n, (4)} \|_2^2 ),
	\end{aligned}
	\label{convergence-3-2}
\end{equation}
\begin{equation}
	\begin{aligned}
		&
		\langle \nabla_h \tilde{\boldsymbol e}^{n, (3)} ,  \nabla_h \tilde{\boldsymbol e}^{n, (4)} \rangle
		\le  \frac12 ( \| \nabla_h \tilde{\boldsymbol e}^{n, (3)} \|_2^2 + \|  \nabla_h \tilde{\boldsymbol e}^{n, (4)} \|_2^2 ) ,\quad \mbox{so that}
		\\
		&
		0.265 \beta k \langle \nabla_h \tilde{\boldsymbol e}^{n, (3)} ,  \nabla_h \tilde{\boldsymbol e}^{n, (4)} \rangle
			\le  0.1325 \beta k ( \| \nabla_h \tilde{\boldsymbol e}^{n, (3)} \|_2^2 + \|  \nabla_h \tilde{\boldsymbol e}^{n, (4)} \|_2^2 ),
	\end{aligned}
	\label{convergence-3-3}
\end{equation}
\begin{equation}
	\begin{aligned}
		&
		\langle { NLE}^{n, (1)}, \tilde{\boldsymbol e}^{n, (4)} \rangle \le
	\frac{\tilde{M}}{2} \| {\underline{\boldsymbol e}}^{n} \|_2^2
			+ \frac{5\beta}{1706} \| \nabla_h {\underline{\boldsymbol e}}^{n}  \|_2^2
			+ (\frac{\tilde{M}}{2}+\frac{853 \tilde{M}^2}{10\beta}) \| \tilde{\boldsymbol e}^{n, (4)} \|_2^2,
		\\
		&
		\langle { NLE}^{n, (2)}, \tilde{\boldsymbol e}^{n, (4)} \rangle \le
\frac{\tilde{M}}{2} \| \tilde{\boldsymbol e}^{n, (2)} \|_2^2
			+ \frac{5\beta}{3071} \| \nabla_h \tilde{\boldsymbol e}^{n, (2)}  \|_2^2
			+ (\frac{\tilde{M}}{2}+\frac{3071\tilde{M}^2}{20\beta}) \| \tilde{\boldsymbol e}^{n, (4)} \|_2^2,
		\\
		&
		\langle { NLE}^{n, (3)}, \tilde{\boldsymbol e}^{n, (4)} \rangle \le
\frac{\tilde{M}}{2} \| \tilde{\boldsymbol e}^{n, (3)} \|_2^2
			+ \frac{\beta}{1100} \| \nabla_h \tilde{\boldsymbol e}^{n, (3)}  \|_2^2
			+ (\frac{\tilde{M}}{2}+\frac{275\tilde{M}^2}{\beta}) \| \tilde{\boldsymbol e}^{n, (4)} \|_2^2,
	\end{aligned}	
	\label{convergence-3-4}
\end{equation}
with the help of the a-priori bound~\eqref{bound-2}~and the regularity estimate~\eqref{consistency-4}. As a consequence, a substitution of~\eqref{convergence-3-2}--\eqref{convergence-3-4} into \eqref{convergence-3-1} leads to

	\begin{equation}
		\begin{aligned}
			&
			\| \tilde{\boldsymbol e}^{n, (4)} \|_2^2 - \| \tilde{\boldsymbol e}^{n, (3)} \|_2^2 + \| \tilde{\boldsymbol e}^{n, (4)} - \tilde{\boldsymbol e}^{n, (3)} \|_2^2
			+ 0.0128 \beta k \| \nabla_h \tilde{\boldsymbol e}^{n, (4)} \|_2^2
			\\
			&
			- 0.3219 \beta k\| \nabla_h \tilde{\boldsymbol e}^{n, (2)} \|_2^2 - 0.1335 \beta k \| \nabla_h \tilde{\boldsymbol e}^{n, (3)} \|_2^2
			\le 0.1706\tilde{M} k \| \underline{\boldsymbol e}^{n} \|_2^2
			\\
			&
			+ 0.001 \beta k \| \nabla_h  \underline{\boldsymbol e}^{n}  \|_2^2+ 0.3071 \tilde{M} k \| \tilde{\boldsymbol e}^{n, (2)} \|_2^2  + 0.55 \tilde{M} k \| \tilde{\boldsymbol e}^{n, (3)} \|_2^2 +  {C}_5  k \| \tilde{\boldsymbol e}^{n, (4)} \|_2^2 , 
		\end{aligned}
		\label{convergence-3-7}
\end{equation}
with ${C}_5=0.3412(\frac{\tilde{M}}{2}+\frac{853 \tilde{M}^2}{10\beta})+0.6142(\frac{\tilde{M}}{2}+\frac{3071 \tilde{M}^2}{20\beta})+1.1(\frac{\tilde{M}}{2}+\frac{275 \tilde{M}^2}{\beta})$. In turn, its combination with~\eqref{convergence-2-6} gives

	\begin{equation}
		\begin{aligned}
			&
			\| \tilde{\boldsymbol e}^{n, (4)} \|_2^2 - \| \underline{\boldsymbol e}^n \|_2^2 + \| \tilde{\boldsymbol e}^{n, (2)} - \underline{\boldsymbol e}^n \|_2^2
			+ \| \tilde{\boldsymbol e}^{n, (3)} - \tilde{\boldsymbol e}^{n, (2)} \|_2^2 +\| \tilde{\boldsymbol e}^{n, (4)} - \tilde{\boldsymbol e}^{n, (3)} \|_2^2
			\\
			&
			+ 0.0662 \beta  k \| \nabla_h \tilde{\boldsymbol e}^{n, (2)} \|_2^2
			+ 0.1044 \beta  k \| \nabla_h \tilde{\boldsymbol e}^{n, (3)} \|_2^2+ 0.0128 \beta k\| \nabla_h \tilde{\boldsymbol e}^{n, (4)} \|_2^2
			\\
			&
			\le  1.25\tilde{M}k\| \underline{\boldsymbol e}^{n}  \|_2^2  + 0.0025\beta k \| \nabla_h  \underline{\boldsymbol e}^{n}  \|_2^2+ (0.45\tilde{M} +{C}_1  )k \|  \tilde{\boldsymbol e}^{n,(2)}  \|_2^2
			\\
			&\quad +  (0.55\tilde{M} +{C}_3)  k \| \tilde{\boldsymbol e}^{n, (3)} \|_2^2  +  {C}_5  k\| \tilde{\boldsymbol e}^{n, (4)} \|_2^2.
		\end{aligned}
		\label{convergence-3-8}
\end{equation} 	

Similarly, with the help of the a-priori estimates~\eqref{bound-2}, and the bound derived  in the first and second RK stages, we arrive at the following rough error estimates at stage 3, by taking~${C}_6 ={C}_1 +{C}_3 +2.25\tilde{M}+0.0025\beta$: 
\begin{align}
	&
	\| \tilde{\boldsymbol e}^{n, (4)} \|_2 \le  \Big( \frac{1+ {C}_6 k }{1 - {C}_5  k } \Big)^\frac12 (k^{\frac{9}{4}} + h^{\frac{13}{4}}) \le 2(k^{\frac{9}{4}} + h^{\frac{13}{4}}), \label{bound-stage-3-1}\\
	& \| \nabla_h \tilde{\boldsymbol e}^{n, (4)} \|_2  \le 9 \beta^{-\frac12} k^{-\frac12}
	\Big( 1+ {C}_6 k \Big)^\frac12 (k^{\frac{9}{4}} + h^{\frac{13}{4}})  \label{bound-stage-3-2}
	\le k^{\frac{7}{4}} + h^{\frac{11}{4}} . 
\end{align}
In turn, by the aid of inverse inequalities, the $\| \cdot \|_\infty$ bound for both the numerical error function $\tilde{\boldsymbol e}^{n, (4)}$ and the numerical solution $\tilde{\boldsymbol m}_4$ could be derived as follows 
\begin{align}
	&
	\| \tilde{\boldsymbol e}^{n, (4)}  \|_{\infty} \le \gamma {h}^{-1/2 }
	( \| \tilde{\boldsymbol e}^{n, (4)}  \|_2 + \| \nabla_h \tilde{\boldsymbol e}^{n, (4)}  \|_2 )
	\le  \gamma \Big( \frac{k^{\frac{7}{4}}}{h^\frac12} + h^{\frac94} \Big) \le \frac18 ,  	
	\label{bound-stage-3-3} 	
	\\
	& 	
	\|  \nabla_h \tilde{\boldsymbol e}^{n, (4)} \|_\infty \le \gamma {h}^{-\frac{3}{2}} \| \nabla_h \tilde{\boldsymbol e}^{n, (4)} \|_2 \le 1 , 	
	\label{bound-stage-3-4} 	
	\\
	&
	\| \tilde{\boldsymbol m}_4 \|_\infty \le \| \tilde{\Phi}^{n, (4)} \|_\infty
	+ \| \tilde{\boldsymbol e}^{n, (4)}  \|_{\infty} \le \frac98 + \frac18 = \frac54 ,  	
	\label{bound-stage-3-5} 	
	\\
	& 	
	\| \nabla_h \tilde{\boldsymbol m}_4 \|_\infty \le \| \nabla_h \tilde{\Phi}^{n, (4)} \|_\infty
	+ \|  \nabla_h \tilde{\boldsymbol e}^{n, (4)} \|_\infty \le { C}^* + 1 = \tilde{ C} . 	
	\label{bound-stage-3-6} 	
\end{align}

\noindent
{\bf Error estimate at Runge-Kutta Stage 4}  \, \, Similar to the previous analysis, taking a discrete inner product with~\eqref{consistency-5-4} by $2 \underline{\boldsymbol e}^{n+1}$ results in

	\begin{equation}
		\begin{aligned}
			&
			\| \underline{\boldsymbol e}^{n+1} \|_2^2 - \| \tilde{\boldsymbol e}^{n, (4)} \|_2^2 + \| \underline{\boldsymbol e}^{n+1} - \tilde{\boldsymbol e}^{n, (4)} \|_2^2
			- 2k \langle \tau ^n ,  \underline{\boldsymbol e}^{n+1} \rangle 
			\\
			&
			=  -0.73 k\langle { NLE}^{n, (2)} ,  \underline{\boldsymbol e}^{n+1} \rangle +0.2638 k \langle { NLE}^{n, (3)} ,  \underline{\boldsymbol e}^{n+1} \rangle
			+ 0.4662 k \langle { NLE}^{n, (4)} ,  \underline{\boldsymbol e}^{n+1} \rangle  .
		\end{aligned}
		\label{convergence-4-1}
\end{equation}
A bound for the local truncation error inner product term is obvious:
\begin{equation}
	\langle \tau ^n ,  \underline{\boldsymbol e}^{n+1} \rangle \le \frac12 (  \| \tau ^n \|_2^2 + \|  \underline{\boldsymbol e}^{n+1}  \|_2^2 ) .
	\label{convergence-4-2}
\end{equation}
Likewise, the estimates for the nonlinear error terms could be similarly performed:
\begin{equation}
	\begin{aligned}
		&
		\langle { NLE}^{n, (2)}, {\underline{\boldsymbol e}}^{n+1} \rangle \le
	\frac{\tilde{M}}{2} \| \tilde{\boldsymbol e}^{n, (2)}  \|_2^2
			+ \frac{\beta}{730} \| \nabla_h \tilde{\boldsymbol e}^{n, (2)}  \|_2^2
			+ (\frac{\tilde{M}}{2}+\frac{365 \tilde{M}^2}{2\beta}) \| {\underline{\boldsymbol e}}^{n+1} \|_2^2,
		\\
		&
		\langle { NLE}^{n, (3)}, {\underline{\boldsymbol e}}^{n+1} \rangle \le
		\frac{\tilde{M}}{2} \| \tilde{\boldsymbol e}^{n, (3)} \|_2^2
			+ \frac{5\beta}{1319} \| \nabla_h \tilde{\boldsymbol e}^{n, (3)}  \|_2^2
			+ (\frac{\tilde{M}}{2}+\frac{1319\tilde{M}^2}{20\beta}) \| {\underline{\boldsymbol e}}^{n+1}\|_2^2,
		\\
		&
		\langle { NLE}^{n, (4)}, {\underline{\boldsymbol e}}^{n+1} \rangle \le
	\frac{\tilde{M}}{2} \| \tilde{\boldsymbol e}^{n, (4)} \|_2^2
			+ \frac{5\beta}{2331} \| \nabla_h \tilde{\boldsymbol e}^{n, (4)}  \|_2^2
			+ (\frac{\tilde{M}}{2}+\frac{2331\tilde{M}^2}{20\beta}) \| {\underline{\boldsymbol e}}^{n+1} \|_2^2.
	\end{aligned}	
	\label{convergence-4-3}
\end{equation}

Subsequently, a substitution of~\eqref{convergence-4-2}--\eqref{convergence-4-3}~into \eqref{convergence-4-1} yields 

	\begin{equation}
		\begin{aligned}
			&
			\| \underline{\boldsymbol e}^{n+1} \|_2^2 - \| \tilde{\boldsymbol e}^{n, (4)} \|_2^2 + \| \underline{\boldsymbol e}^{n+1} - \tilde{\boldsymbol e}^{n, (4)} \|_2^2
			- 0.001 \beta k \| \nabla_h \tilde{\boldsymbol e}^{n, (2)} \|_2^2
			\\
			&
			- 0.001 \beta k \| \nabla_h \tilde{\boldsymbol e}^{n, (3)} \|_2^2
			- 0.001 \beta k \| \nabla_h \tilde{\boldsymbol e}^{n, (4)} \|_2^2 \le 0.365\tilde{M} k  \|  \tilde{\boldsymbol e}^{n, (2)}  \|_2^2
			\\
			&
			+ 0.1319 \tilde{M} k \|  \tilde{\boldsymbol e}^{n, (3)}  \|_2^2
			+ 0.2331 \tilde{M} k \|  \tilde{\boldsymbol e}^{n, (4)}  \|_2^2
			+ {C}_7 k \| \underline{\boldsymbol e}^{n+1}  \|_2^2
			+ k (  \| \tau^n \|_2^2 + \|  \underline{\boldsymbol e}^{n+1}  \|_2^2 ) , 
		\end{aligned}
		\label{convergence-4-6}
\end{equation}
with~${C}_7=0.73(\frac{\tilde{M}}{2}+\frac{365\tilde{M}^2}{2\beta})+0.2638(\frac{\tilde{M}}{2}+\frac{1319 \tilde{M}^2}{20\beta})+0.4662(\frac{\tilde{M}}{2}+\frac{2331 \tilde{M}^2}{20\beta})$. Its combination with~\eqref{convergence-3-8} leads to 
	\begin{equation}
		\begin{aligned}
			&
			\| \underline{\boldsymbol e}^{n+1} \|_2^2 - \| \underline{\boldsymbol e}^n \|_2^2 + \| \tilde{\boldsymbol e}^{n, (2)} - \underline{\boldsymbol e}^n \|_2^2
			+ \| \tilde{\boldsymbol e}^{n, (3)} - \tilde{\boldsymbol e}^{n, (2)} \|_2^2 + \| \underline{\boldsymbol e}^{n+1} - \tilde{\boldsymbol e}^{n, (4)} \|_2^2
			\\
			& 
			+ 0.0652 \beta k  \| \nabla_h \tilde{\boldsymbol e}^{n, (2)} \|_2^2 + 0.1034\beta k  \| \nabla_h \tilde{\boldsymbol e}^{n, (3)} \|_2^2 + 0.0118 \beta k  \| \nabla_h \tilde{\boldsymbol e}^{n, (4)} \|_2^2
			\\
			&
			\le 1.25\tilde{M}k \| \underline{\boldsymbol e}^{n}  \|_2^2  + 0.0025\beta k \| \nabla_h  \underline{\boldsymbol e}^{n}  \|_2^2  + (0.815\tilde{M}+{C}_1)  k \| \tilde{\boldsymbol e}^{n, (2)}  \|_2^2+(0.6819\tilde{M}+{C}_3)  k
			\\
			&
			\| \tilde{\boldsymbol e}^{n, (3)}  \|_2^2
			+(0.2331\tilde{M} +{C}_5)  k \| \tilde{\boldsymbol e}^{n, (4)}  \|_2^2
			+{C}_7  k\| {\underline{\boldsymbol e}}^{n+1}  \|_2^2
			+  k (  \| \tau ^n \|_2^2 + \|  \underline{\boldsymbol e}^{n+1}  \|_2^2 ).
		\end{aligned}
		\label{convergence-4-7}
\end{equation}
Meanwhile, an application of triangular inequality indicates that
\begin{equation}
	\begin{aligned}
		&
		\| \underline{\boldsymbol e}^{n+1}  \|_2  \le  \| \tilde{\boldsymbol e}^{n, (4)} \|_2 + \| \underline{\boldsymbol e}^{n+1} - \tilde{\boldsymbol e}^{n, (4)} \|_2 ,
		\quad \mbox{so that}
		\\
		&
		{C}_7  k \| \underline{\boldsymbol e}^{n+1}  \|_2^2  \le  2 {C}_7  k ( \| \tilde{\boldsymbol e}^{n, (4)} \|_2^2
		+ \| \underline{\boldsymbol e}^{n+1} - \tilde{\boldsymbol e}^{n, (4)} \|_2^2 ) ,
		\quad \mbox{and}
		\\
		&
		2 {C}_7 k \| \underline{\boldsymbol e}^{n+1} - \tilde{\boldsymbol e}^{n, (4)} \|_2^2
		\le  \frac12 \|\underline{\boldsymbol e}^{n+1} - \tilde{\boldsymbol e}^{n, (4)}  \|_2^2 ,  \quad \mbox{provided that $ {C}_7  k \le \frac14$,}
	\end{aligned}
	\label{convergence-4-8}
\end{equation}
which is always valid under the linear refinement requirement, $C_1 h \le k \le C_2 h$, and the assumption that $k$ and $h$ are sufficiently small. Therefore, a substitution of~\eqref{convergence-4-8} into~\eqref{convergence-4-7} results in 
	\begin{equation}
		\begin{aligned}
			&
			\| \underline{\boldsymbol e}^{n+1} \|_2^2 - \| \underline{\boldsymbol e}^n \|_2^2 + \| \tilde{\boldsymbol e}^{n, (2)} - \underline{\boldsymbol e}^n \|_2^2
			+ \| \tilde{\boldsymbol e}^{n, (3)} - \tilde{\boldsymbol e}^{n, (2)} \|_2^2  + \| \tilde{\boldsymbol e}^{n, (4)} - \tilde{\boldsymbol e}^{n, (3)} \|_2^2
			\\
			&
			+ \frac12 \| \underline{\boldsymbol e}^{n+1} - \tilde{\boldsymbol e}^{n, (4)} \|_2^2
			+ 0.0652 \beta k  \| \nabla_h \tilde{\boldsymbol e}^{n, (2)} \|_2^2
			+ 0.1034\beta k  \| \nabla_h \tilde{\boldsymbol e}^{n, (3)} \|_2^2 
			\\
			&
			+ 0.0118 \beta k  \| \nabla_h \tilde{\boldsymbol e}^{n, (4)} \|_2^2\le 1.25\tilde{M}k \| \underline{\boldsymbol e}^{n}  \|_2^2  + 0.0025\beta k \| \nabla_h  \underline{\boldsymbol e}^{n}  \|_2^2  
			\\
			&
			+
			{C}_8  k \| \tilde{\boldsymbol e}^{n, (2)}  \|_2^2+{C}_9  k \| \tilde{\boldsymbol e}^{n, (3)}  \|_2^2 +{C}_{10}  k \| \tilde{\boldsymbol e}^{n, (4)}  \|_2^2
			+  k (  \| \tau ^n \|_2^2 + \|  \underline{\boldsymbol e}^{n+1}  \|_2^2 ) , 
		\end{aligned}
		\label{convergence-4-9}
\end{equation}
with~${C}_8=0.815\tilde{M}+{C}_1$, ${C}_9=0.6819\tilde{M} +{C}_3$, and ${C}_{10}=2{C}_7+0.2331\tilde{M} +{C}_5$. 

However, the standard $\ell^2$ error estimate~\eqref{convergence-4-9} does not allow one to apply discrete Gronwall inequality, due to the $H_{h}^{1}$ norms of the error function involved on the right-hand side. To overcome this difficulty, we apply the gradient operation on both sides of~\eqref{consistency-5-4}, with the linear refinement requirement~$k \le C^{'} h$, and see that 
	\begin{equation}
		\begin{aligned}
			& \| \nabla_h \underline{\boldsymbol e}^{n+1} \|_2 \le \| \nabla_h \tilde{\boldsymbol e}^{n, (4)} \|_2 + k \| \nabla_h \tau ^n \|_2+ 0.365 k\| \nabla_h { NLE}^{n, (2)} \|_2
			\\
			&+ 0.1319k\| \nabla_h { NLE}^{n, (3)} \|_2  + 0.2331 k  \| \nabla_h { NLE}^{n, (4)} \|_2 ,
			\\
			& \le \| \nabla_h \tilde{\boldsymbol e}^{n, (4)} \|_2 + C^{'} \|  \tau ^n \|_2+ 0.365C^{'} \|  { NLE}^{n, (2)} \|_2 
			\\
			&+ 0.1319 C^{'} \|  { NLE}^{n, (3)} \|_2 + 0.2331 C^{'} \|  { NLE}^{n, (4)} \|_2,
			\\
			&\le \| \nabla_h \tilde{\boldsymbol e}^{n, (4)} \|_2+  C^{'} \|  \tau ^n \|_2+C^{'}\tilde{M} \Big( 0.365(\| \tilde{\boldsymbol e}^{n, (2)}\|_2+\|\nabla_h \tilde{\boldsymbol e}^{n, (2)}\|_2)\Big)
			\\
			& +C^{'}\tilde{M} \Big(0.1319 (\| \tilde{\boldsymbol e}^{n, (3)}\|_2+\|\nabla_h \tilde{\boldsymbol e}^{n, (3)}\|_2) + 0.2331 (\| \tilde{\boldsymbol e}^{n, (4)}\|_2+\|\nabla_h \tilde{\boldsymbol e}^{n, (4)}\|_2) \Big).
		\end{aligned}
		\label{convergence-4-10}
\end{equation}
Meanwhile, the following result could be derived at the previous time step:
	\begin{equation}
		\begin{aligned}
			& \| \nabla_h \underline{\boldsymbol e}^{n} \|_2^2
			\le 2C^{'} \|  \tau ^{n-1} \|_2^2  +  0.73C^{'}\tilde{M}  (\| \tilde{\boldsymbol e}^{n-1, (2)}\|_2^2+\|\nabla_h \tilde{\boldsymbol e}^{n-1, (2)}\|_2^2) 
			+ 0.2638C^{'}\tilde{M}\\
			&
			(\| \tilde{\boldsymbol e}^{n-1, (3)}\|_2^2+\|\nabla_h \tilde{\boldsymbol e}^{n-1, (3)}\|_2^2)+ 0.4662C^{'}\tilde{M} \| \tilde{\boldsymbol e}^{n-1, (4)}\|_2^2 +2(0.2331C^{'}\tilde{M}+1)\|\nabla_h \tilde{\boldsymbol e}^{n-1, (4)}\|_2^2 .
		\end{aligned}
		\label{convergence-4-11}
\end{equation}
In turn, a substitution of~\eqref{convergence-4-11} into~\eqref{convergence-4-9} yields 

	\begin{equation}
		\begin{aligned}
			&
			\| \underline{\boldsymbol e}^{n+1} \|_2^2 - \| \underline{\boldsymbol e}^n \|_2^2 + \| \tilde{\boldsymbol e}^{n, (2)} - \underline{\boldsymbol e}^n \|_2^2
			+ \| \tilde{\boldsymbol e}^{n, (3)} - \tilde{\boldsymbol e}^{n, (2)} \|_2^2
			\\
			&
			+ \| \tilde{\boldsymbol e}^{n, (4)} - \tilde{\boldsymbol e}^{n, (3)} \|_2^2+ \frac12 \| \underline{\boldsymbol e}^{n+1} - \tilde{\boldsymbol e}^{n, (4)} \|_2^2
			+ 0.0652 \beta k  \| \nabla_h \tilde{\boldsymbol e}^{n, (2)} \|_2^2
			\\
			&
			+ 0.1034\beta k  \| \nabla_h \tilde{\boldsymbol e}^{n, (3)} \|_2^2 + 0.0118 \beta k  \| \nabla_h \tilde{\boldsymbol e}^{n, (4)} \|_2^2
			\\
			&
			\le 1.25\tilde{M}k\| \underline{\boldsymbol e}^{n}  \|_2^2  + 0.0025\beta k \big(\gamma_1\| \tilde{\boldsymbol e}^{n-1, (2)}\|_2^2+
			\gamma_2\| \tilde{\boldsymbol e}^{n-1, (3)}\|_2^2 +\gamma_3\| \tilde{\boldsymbol e}^{n-1, (4)}\|_2^2
			\\
			&
			+\gamma_1\| \nabla_h \tilde{\boldsymbol e}^{n-1, (2)}\|_2^2+
			\gamma_2\| \nabla_h \tilde{\boldsymbol e}^{n-1, (3)}\|_2^2 +\gamma_3\| \nabla_h \tilde{\boldsymbol e}^{n-1, (4)}\|_2^2\big)+{C}_8  k \| \tilde{\boldsymbol e}^{n, (2)}  \|_2^2
			\\
			&
			+{C}_9  k \| \tilde{\boldsymbol e}^{n, (3)}  \|_2^2 +{C}_{10}  k \| \tilde{\boldsymbol e}^{n, (4)}  \|_2^2
			+  k \|  \underline{\boldsymbol e}^{n+1}  \|_2^2 +{C}_{11} k (\| \tau ^n \|_2^2+\| \tau ^{n-1} \|_2^2 ) , 
		\end{aligned}
		\label{convergence-4-12}
\end{equation}

	\begin{equation}
		\begin{aligned}
			&
			\| \underline{\boldsymbol e}^{n+1} \|_2^2 - \| \underline{\boldsymbol e}^n \|_2^2 + \| \tilde{\boldsymbol e}^{n, (2)} - \underline{\boldsymbol e}^n \|_2^2
			+ \| \tilde{\boldsymbol e}^{n, (3)} - \tilde{\boldsymbol e}^{n, (2)} \|_2^2+ \| \tilde{\boldsymbol e}^{n, (4)} - \tilde{\boldsymbol e}^{n, (3)} \|_2^2
			\\
			&
			+ \frac12 \| \underline{\boldsymbol e}^{n+1} - \tilde{\boldsymbol e}^{n, (4)} \|_2^2
			+ 0.0652 \beta k  \| \nabla_h \tilde{\boldsymbol e}^{n, (2)} \|_2^2+ 0.1034\beta k  \| \nabla_h \tilde{\boldsymbol e}^{n, (3)} \|_2^2
			\\
			&
			+ 0.0118 \beta k  \| \nabla_h \tilde{\boldsymbol e}^{n, (4)} \|_2^2+k(\| \nabla_h \underline{\boldsymbol e}^{n+1} \|_2^2-\| \nabla_h \underline{\boldsymbol e}^{n} \|_2^2)
			\\
			&
			\le 1.25\tilde{M}k \| \underline{\boldsymbol e}^{n}  \|_2^2  + 0.0025\beta k \big(\gamma_1\| \tilde{\boldsymbol e}^{n-1, (2)}\|_2^2+
			\gamma_2\| \tilde{\boldsymbol e}^{n-1, (3)}\|_2^2 +\gamma_3\| \tilde{\boldsymbol e}^{n-1, (4)}\|_2^2
			\\
			&
			+\gamma_1\| \nabla_h \tilde{\boldsymbol e}^{n-1, (2)}\|_2^2+
			\gamma_2\| \nabla_h \tilde{\boldsymbol e}^{n-1, (3)}\|_2^2 +\gamma_3\| \nabla_h \tilde{\boldsymbol e}^{n-1, (4)}\|_2^2\big) +k\big(  \gamma_1\|\nabla_h \tilde{\boldsymbol e}^{n, (2)}\|_2^2
			\\
			&
			+ \gamma_2\|\nabla_h \tilde{\boldsymbol e}^{n, (3)}\|_2^2+\gamma_3\|\nabla_h \tilde{\boldsymbol e}^{n, (4)}\|_2^2    \big)+{C}_{12}   k \| \tilde{\boldsymbol e}^{n, (2)}  \|_2^2+{C}_{13}   k \| \tilde{\boldsymbol e}^{n, (3)}  \|_2^2
			\\
			&
			+{C}_{14}   k \| \tilde{\boldsymbol e}^{n, (4)}  \|_2^2
			+  k \|  \underline{\boldsymbol e}^{n+1}  \|_2^2 +{C}_{15} k (\| \tau ^n \|_2^2+\| \tau ^{n-1} \|_2^2 ).
		\end{aligned}
		\label{convergence-4-13}
\end{equation}
with $\gamma_1 =0.73C^{'}\tilde{M}$, $\gamma_2 = 0.2638C^{'}\tilde{M}$ and $\gamma_3 =2+ 0.4662C^{'}\tilde{M}$. Therefore, with the help of the triangular inequalities: 
\begin{equation}
	\begin{aligned}
		&
		\|  \tilde{\boldsymbol e}^{n,(2)}  \|_2 \le \| \underline{\boldsymbol e}^{n} \|_2 + \| \tilde{\boldsymbol e}^{n, (2)} - \underline{\boldsymbol e}^{n} \|_2  ,
		\\
		&
		\|  \tilde{\boldsymbol e}^{n,(3)}  \|_2 \le \| \underline{\boldsymbol e}^{n} \|_2 + \| \tilde{\boldsymbol e}^{n, (2)} - \underline{\boldsymbol e}^{n} \|_2
		+ \| \tilde{\boldsymbol e}^{n, (3)} - \tilde{\boldsymbol e}^{n, (2)} \|_2 ,
		\\
		&
		\|  \tilde{\boldsymbol e}^{n,(4)}  \|_2 \le \| \underline{\boldsymbol e}^{n} \|_2 + \| \tilde{\boldsymbol e}^{n, (2)} - \underline{\boldsymbol e}^{n} \|_2
		+ \| \tilde{\boldsymbol e}^{n, (3)} - \tilde{\boldsymbol e}^{n, (2)} \|_2
		+ \| \tilde{\boldsymbol e}^{n, (4)} - \tilde{\boldsymbol e}^{n, (3)} \|_2 , 
	\end{aligned}
	\label{convergence-4-14}
\end{equation}
we get the following estimate:

	\begin{equation}
		\begin{aligned}
			&
			\| \underline{\boldsymbol e}^{n+1} \|_2^2 - \| \underline{\boldsymbol e}^n \|_2^2+k(\| \nabla_h \underline{\boldsymbol e}^{n+1} \|_2^2-\| \nabla_h \underline{\boldsymbol e}^{n} \|_2^2)
			+ (0.0652 \beta-\gamma1) k  \| \nabla_h \tilde{\boldsymbol e}^{n, (2)} \|_2^2
			\\
			&
			+ (0.1034\beta-\gamma2) k  \| \nabla_h \tilde{\boldsymbol e}^{n, (3)} \|_2^2
			+ (0.0118 \beta-\gamma3) k  \| \nabla_h \tilde{\boldsymbol e}^{n, (4)} \|_2^2
			\\
			&
			\le {C}k( \| \underline{\boldsymbol e}^{n-1}  \|_2^2+\| \underline{\boldsymbol e}^{n}  \|_2^2+\| \underline{\boldsymbol e}^{n+1}  \|_2^2  )+{C}_{15} k (\| \tau ^n \|_2^2+\| \tau ^{n-1} \|_2^2 ),
			\\
			&
			+0.0025\beta k \big(\gamma_1\| \nabla_h \tilde{\boldsymbol e}^{n-1, (2)}\|_2^2+
			\gamma_2| \nabla_h \tilde{\boldsymbol e}^{n-1, (3)}\|_2^2 +\gamma_3\| \nabla_h \tilde{\boldsymbol e}^{n-1, (4)}\|_2^2\big).
		\end{aligned}
		\label{convergence-4-15}
\end{equation}
Finally, an application of discrete Gronwall inequality \cite{Girault1986} leads to the desired error estimate at the next time step: 
\begin{equation}
	\| \underline{\boldsymbol e}^{n+1} \|_2 + (k \| \nabla_h \underline{\boldsymbol e}^{n+1} \|_2)^\frac12 \le {C} (k^3+h^4) ,  \label{convergence-22} 	
\end{equation}
which comes from the fact that $652\beta >(25\beta+10000)\gamma_{1}, 1034\beta >(25\beta+10000)\gamma_{2},\\ 118\beta >(25\beta+10000)\gamma_{3}$, where $\gamma_1=0.73C^{'}\tilde{M}, \gamma_2=0.2638C^{'}\tilde{M}, \gamma_3=2+0.4662C^{'}\tilde{M}$. In particular, the local truncation error estimate, $\| \tau ^n \|^2 , \| \tau ^{n-1} \|^2 \le { C} (k^3 + h^4)$, was used in the derivation.

As a result, we see that the a-priori assumption~\eqref{bound-2} has also been validated at the next time step $t^{n+1}$, provided that $k$ and $h$ are sufficiently small.

By a mathematical induction argument, the higher order error estimate~\eqref{convergence-22} is valid for any time step. Of course, the convergence estimate~\eqref{convergence-0} becomes a direct consequence of the following identity:
\begin{equation}
	\Phi^n - \boldsymbol m^n = \underline{\boldsymbol e}^n - h^2 \Phi^{(1)} ,  \label{convergence-23} 	
\end{equation}
which comes from the constructed profile $\underline{\Phi}^n = \Phi^n + h^2 \Phi^{(1), n}$. The proof of Theorem~\ref{thm: convergence} is completed.

\section{Numerical results} \label{section:numercial tests}
In this section, we perform 1D and 3D numerical experiments to verify the theoretical canalysis in Section~\ref{sec: convergence}. For simplicity, we set $\epsilon=1$, $\boldsymbol{f}=0$ in \eqref{eq-9}, and $\alpha=0.01$, $\beta=3$ in the next accuracy test. The 1-D exact solution is taken to be
$$
\boldsymbol{m}_{e}=(\cos (X) \sin t, \sin (X) \sin t, \cos t)^{T}, \quad \mbox{with} \, \, \,  
X=x^{2}(1-x)^{2} . 
$$
The 3-D exact solution is chosen to be 
$$
\boldsymbol{m}_{e}=(\cos (X Y Z) \sin t, \sin (X Y Z) \sin t, \cos t)^{T},
$$
where $X=x^{2}(1-x)^{2}, Y=y^{2}(1-y)^{2}, Z=z^{2}(1-z)^{2}$. Clearly the homogeneous Neumann boundary condition \eqref{eq-2} is satisfied and a forcing term $\boldsymbol{f}_{e}=\partial_{t} \boldsymbol{m}_{e}-\alpha \Delta \boldsymbol{m}_{e}-\alpha\left|\nabla \boldsymbol{m}_{e}\right|^{2}+\boldsymbol{m}_{e} \times$ $\Delta \boldsymbol{m}_{e}$ is included into the nonlinear part $N(t,\boldsymbol{m})$.

\subsection{Accuracy test of IMEX-RK3}
In the 1-D computation, we fix $k =0.0001\times {{h}^{\frac{2}{3}}}$ and record the error in terms of $h$ in Table~\ref{table1}, fix $k=(1e-03)/(1e+04)$ and record the error in terms of $h$ in Table~\ref{table2}.
\begin{table}[htbp]
	\centering
	\caption{Temporal accuracy check in the 1-D case ($k =0.0001\times {{h}^{\frac{2}{3}}}$).}\label{table1}
		\begin{tabular}{cccc}
			\hline$k$ & $\left\|\boldsymbol m_{h}-\boldsymbol m_{e}\right\|_{\infty}$ & $\left\|\boldsymbol m_{h}-\boldsymbol m_{e}\right\|_{2}$ & $\left\|\boldsymbol m_{h}-\boldsymbol m_{e}\right\|_{H^{1}}$ \\
			\hline
			$0.1 / 3302$ & $1.8064e-04$ & $1.8228e-04$ & $2.4995e-03 $ \\
			$0.1 / 3659$ & $1.3617e-04$ & $1.3515e-04$ & $1.8393e-03$ \\
			$0.1 / 4000$ & $1.0275e-04$ & $1.0401e-04$ & $1.4117e-03$ \\
			$0.1 / 4327$ & $8.1831e-05$ & $8.2470e-05$ & $1.1171e-03$ \\
			order & $2.9492$ & $2.9336$ & $2.9782$ \\
			\hline
	\end{tabular}
\end{table}
\begin{table}[htbp]
	\centering
	\caption{Spatial accuracy check in the 1-D case ($k=(1e-03)/(1e+04)$).}\label{table2}
		\begin{tabular}{cccc}
			\hline$h$ & $\left\|\boldsymbol m_{h}-\boldsymbol m_{e}\right\|_{\infty}$ & $\left\|\boldsymbol m_{h}-\boldsymbol m_{e}\right\|_{2}$ & $\left\|\boldsymbol m_{h}-\boldsymbol m_{e}\right\|_{H^{1}}$ \\
			\hline
			$1 / 160 $ & $2.8966e-10$ & $8.5153e-11$ & $4.4271e-08$ \\
			$1 / 240 $ & $1.2934e-10$ & $3.7953e-11$ & $1.9679e-08$ \\
			$1 / 320 $ & $7.2876e-11$ & $2.1370e-11$ & $1.1070e-08$ \\
			$1 / 400 $ & $4.6676e-11$ & $1.3683e-11$ & $7.0848e-09$ \\
			order & $1.9922$ & $1.9953$ & $1.9998$ \\
			\hline
	\end{tabular}
\end{table}

\begin{table}[htbp]
	\centering
	\caption{Temporal accuracy check in the 3-D case ($k =0.001\times {{h}^{\frac{2}{3}}}$).}\label{table3}
		\begin{tabular}{cccc}
			\hline$k$ & $\left\|\boldsymbol m_{h}-\boldsymbol m_{e}\right\|_{\infty}$ & $\left\|\boldsymbol m_{h}-\boldsymbol m_{e}\right\|_{2}$ & $\left\|\boldsymbol m_{h}-\boldsymbol m_{e}\right\|_{H^{1}}$ \\
			\hline
			$1 / 1587$ & $3.5857e-04$ & $2.4600e-04$ & $4.4100e-04$ \\
			$1 / 2080$ & $1.5051e-04$ & $1.0164e-04$ & $2.0036e-04$ \\
			$1 / 2520$ & $8.1408e-05$ & $5.7072e-05$ & $1.0807e-04$ \\
			$1 / 2924$ & $5.4348e-05$ & $3.7012e-05$ & $6.5389e-05$ \\
			order & $3.1103$ & $3.1020$ & $3.1180$ \\
			\hline
	\end{tabular}
\end{table}
In the 3-D computation, we also fix $k =0.001\times {{h}^{\frac{2}{3}}}$ and record the error in terms of $k$ in Table~\ref{table3}, fix $k=1/10000$ and record the error in terms of $h$ in Table~\ref{table4}.

\begin{table}[htbp]
	\centering
	\caption{Spatial accuracy check in the 3-D case ($k=1e-04$).}\label{table4}
		\begin{tabular}{cccc}
			\hline$h$ & $\left\|\boldsymbol m_{h}-\boldsymbol m_{e}\right\|_{\infty}$ & $\left\|\boldsymbol m_{h}-\boldsymbol m_{e}\right\|_{2}$ & $\left\|\boldsymbol m_{h}-\boldsymbol m_{e}\right\|_{H^{1}}$ \\
			\hline
			$1 / 4$ & $8.1432e-05$ & $5.7082e-05$ & $9.3202e-05$ \\
			$1 / 5$ & $5.4354e-05$ & $3.7020e-05$ & $6.1874e-05$ \\
			$1 / 6$ & $3.6180e-05$ & $2.6471e-05$ & $4.0970e-05$ \\
			$1 / 7$ & $2.7160e-05$ & $1.8529e-05$ & $3.0914e-05$ \\
			order & $1.9861$ & $1.9881$ & $1.9987$ \\
			\hline
	\end{tabular}
\end{table}

\begin{figure}[htbp]
	\centering
{\label{time_1D}\includegraphics[width=2.5in]{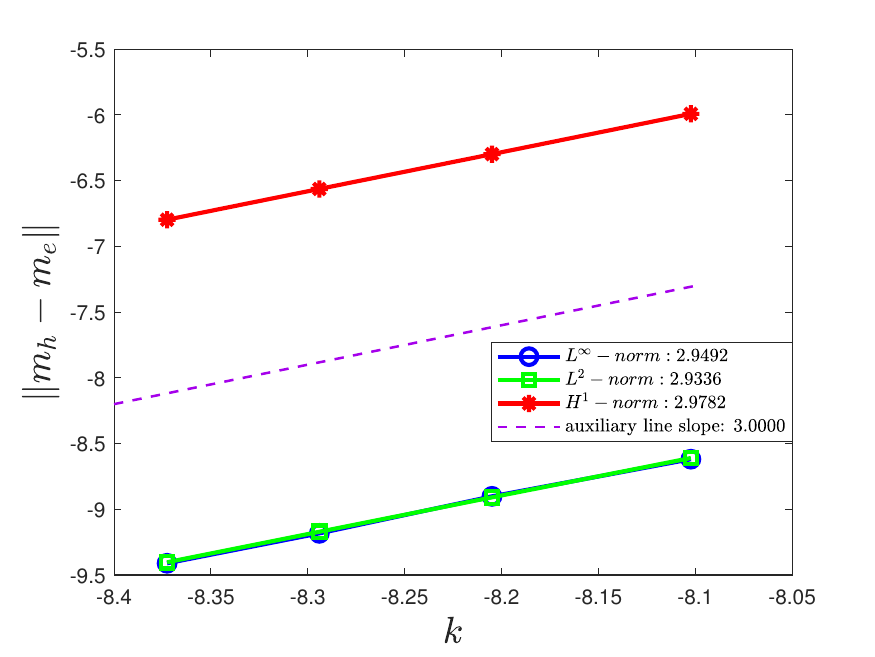}}
{\label{space_1D}\includegraphics[width=2.5in]{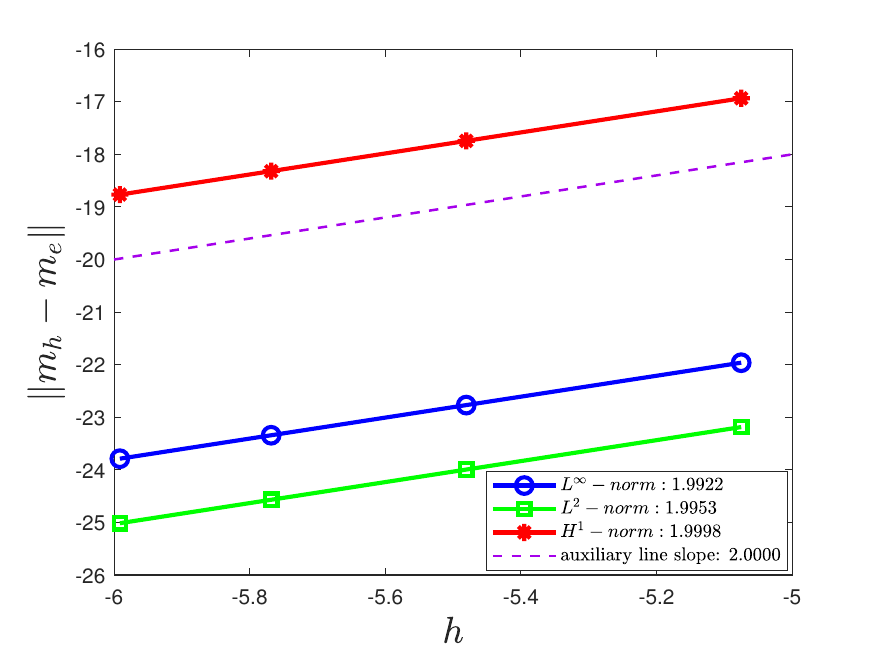}}
	\quad	
{\label{time_3D}\includegraphics[width=2.5in]{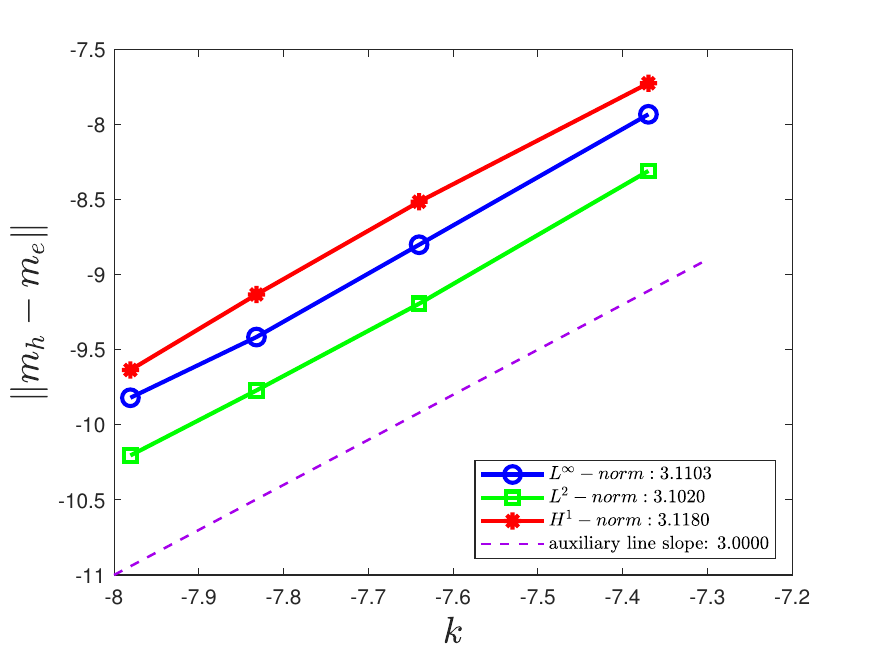}}
{\label{space_3D}\includegraphics[width=2.5in]{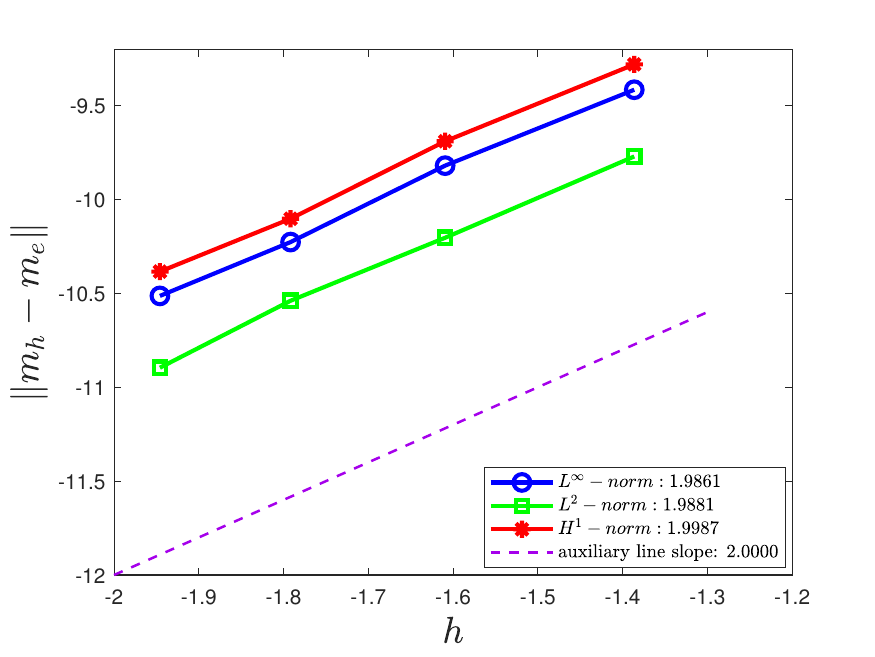}}
	\caption{Temporal and spatial accuracy orders in the 1-D and 3-D domains computations. Top row: 1-D; Bottom row: 3-D.}
	\label{rk2 cpu time}
\end{figure}

\subsection{Dependence on the damping parameter}
The GSPM method \cite{wang2001gauss} is unconditionally stable with constant coefficients and SPD structure, while its primary disadvantage is associated with its first-order accuracy in time. The SIPM \cite{xie2020second,chen2021convergence} is indeed a second-order-in-time method, while the non-symmetric structure and variable coefficients have led to more expensive computational costs. In addition, the two above-mentioned methods have focused on small damping parameter, nevertheless, large damping parameter has also been considered in the numerical design for real micromagnetics in general. Afterwards, the SIPM scheme (with large damping) in \cite{cai2022second} has greatly improved the computational efficiency, since only three Poisson solvers are needed at each time step. Meanwhile, this numerical approach only works if $\alpha>1$, while most magnetic materials correspond to $\alpha\ll 1$. On the other hand, for the BDF schemes of orders 3 to 5~\cite{Lubich2021}, coupled with higher-order finite element spatial discretization, a positive lower bound on the damping $\alpha$ is needed to ensure a numerical stability. In more details, the damping parameter satisfies $\alpha > \alpha_k$ with $\alpha_k=0.0913, 0.4041, 4.4348$ for orders $k=3,4,5$, respectively, for the BDF-k method analyzed in \cite{Lubich2021}. Therefore, it is worthwhile to design an efficient and higher order accurate numerical scheme that is unconstrained by the damping parameter $\alpha$.

To investigate the dependence on the damping parameter for the proposed IMEX-RK3 scheme, two different damping parameters, $\alpha=0.01, 0.1$, are taken, and $k =0.0001 \times {{h}^{\frac{2}{3}}}$ is fixed. The results of 1-D and 3-D corresponding examples are presented in Table~\ref{table6} and Table~\ref{table7}. It is observed that the choice of $\alpha$ is arbitrary, and the third-order accuracy is preserved in the temporal discretization. Based upon these results, it is clear that the proposed IMEX-RK3 method works well for general artificial damping parameters. More comparison results and details are displayed in Table~\ref{table5}. In fact, we set $\beta > 1$ if $\alpha\ll 1$, and $\beta=\alpha$ if $\alpha \ge 1$, then apply the IMEX-RK numerical scheme. As a result, the proposed numerical method works for a general damping parameter.

\begin{table}[htbp]
		\centering
	\caption{Comparison between the GSPM, BDF3, SIPM, SIPM with large damping and the IMEX-RK3 proposed scheme.}\label{table5}
	\begin{tabular}{cccc}
		\hline Property or number & Scope of $\alpha$ & Symmetry & Accuracy in time \\
		\hline GSPM   & not arbitrary & Yes & ${O}\left(k\right)$ \\
		BDF3    & $\alpha>0.0913$ & No & ${O}\left(k^3\right)$ \\
		SIPM    & arbitrary  & No & ${O}\left(k^2\right)$ \\
		SIPM with large damping   & $\alpha>1$ & Yes & ${O}\left(k^2\right)$ \\
		IMEX-RK3 proposed & arbitrary & Yes & ${O}\left(k^3\right)$ \\
		\hline
	\end{tabular}
\end{table}

\begin{table}[htbp]
	\centering
	\small
	\renewcommand{\arraystretch}{1.2}
	\setlength{\tabcolsep}{0.4mm}{
		\caption{1-D numerical errors of the IMEX-RK3 scheme.}\label{table6}
		\begin{tabular}{|c|c|ccc|ccc|}
			\hline  \multirow{2}{*}{$\beta$}&\multirow{2}{*} {$k$} & \multicolumn{3}{c|}{$\alpha=0.001$} & \multicolumn{3}{c|}{$\alpha=0.1$} \\
			\cline { 3 - 8 } & & $L^{\infty}$ & $L^{2}$ & $H^{1}$ & $L^{\infty}$  & $L^{2}$ & $H^{1}$ \\
			
			\hline                & 	$0.1/3302$   & $1.8127e-04$ & $1.8294e-04$ & $2.4964e-03$   & $1.7463e-04$ & $1.7719e-04$ & $2.5257e-03$\\
			
			$1$  & $0.1/3659$    & $1.3680e-04$ &$1.3565e-04$ & $1.8368e-03$  & $1.3054e-04$ & $1.3129e-04$ & $1.8604e-03$\\
			& $0.1/4000$   &$1.0323e-04$ &$ 1.0439e-04$ & $1.4099e-03$   &$9.8705e-05$ &$1.0104e-04$ & $1.4273e-03$ \\\hline
			\multicolumn{2}{|c|}{\text { order }}
			& $2.9311 $             & $2.9253$&$2.9796 $ &$2.9715 $            & $2.9289$&$2.9762$ \\
			
			\hline                & 	$0.1/3302$   & $1.8127e-04$ & $1.8294e-04$ &$2.4964e-03$   & $1.7463e-04$ & $1.7719e-04$ & $2.5257e-03$\\
			
			$3$  & $0.1/3659$    & $1.3680e-04$& $1.3565e-04$ & $1.8368e-03$  & $1.3054e-04$ & $1.3129e-04$ & $1.8604e-03$\\
			& $0.1/4000$   & $1.0323e-04$ &$ 1.0439e-04$ & $1.4099e-03$  &$9.8705e-05$ &$1.0104e-04$& $1.4273e-03$\\\hline
			\multicolumn{2}{|c|}{\text { order }}
			&$2.9311 $             & $2.9253$&$2.9796 $ &$2.9715 $             & $2.9289$&$2.9762$
			\\
			\hline                            & 	$0.1/3302$   & $1.8127e-04$ & $1.8294e-04$ & $2.4964e-03$  &$1.7463e-04$& $1.7719e-04$ & $2.5257e-03$\\
			
			$5$ & $0.1/3659$   & $1.3680e-04$ & $1.3565e-04$ & $1.8368e-03$   & $1.3054e-04$ & $1.3129e-04$& $1.8604e-03$\\
			& $0.1/4000$   & $1.0323e-04$ &$ 1.0439e-04$ & $1.4099e-03$   &$9.8705e-05$ &$1.0104e-04$ & $1.4273e-03$ \\\hline
			\multicolumn{2}{|c|}{\text { order }}
			& $2.9311 $            & $2.9253$&$2.9796 $&$2.9715 $            & $2.9289$&$2.9762$ 
			\\
			\hline
	\end{tabular}}
\end{table}

\begin{table}[htbp]
	\centering
	\small
	\renewcommand{\arraystretch}{1.2}
	\setlength{\tabcolsep}{0.4mm}{
		\caption{3-D numerical errors of the IMEX-RK3 scheme.}\label{table7}
		\begin{tabular}{|c|c|ccc|ccc|}
			\hline \multirow{2}{*}{  $\beta$} & \multirow{2}{*}{$k$} & \multicolumn{3}{c|}{$\alpha=0.001$} & \multicolumn{3}{c|}{$\alpha=0.1$} \\
			\cline { 3 - 8 } & & $L^{\infty}$ & $L^{2}$ & $H^{1}$ & $L^{\infty}$  & $L^{2}$ & $H^{1}$ \\
			
			\hline                & $1/2080$    & $1.5087e-04$ & $1.0163e-04$ & $2.0010e-04$   & $ 1.5183e-04$& $1.0209e-04 $& $2.0010e-04$\\
			
			$1$  & $1/2520$   & $ 8.1578e-05$ & $5.7080e-05$& $1.0798e-04$   & $ 8.3704e-05$ &$  5.7304e-05$ & $1.0842e-04$\\
			
			& $1/2924$   & $5.3962e-05$ &$3.6999e-05$ & $6.5334e-05$   &$5.2238e-05 $& $ 3.6101e-05$ & $6.5537e-05$ \\\hline
			\multicolumn{2}{|c|}{\text { order }}
			& $3.0275 $             & $2.9688 $&$3.2830$ &$ 3.1313 $&$ 3.0501$& $3.2733$ \\
			
			\hline                & $1/2080$    & $1.5087e-04$ & $1.0163e-04$ & $2.0010e-04$   & $ 1.5183e-04$ & $1.0209e-04 $& $2.0010e-04$\\
			
			$3$  & $1/2520$   & $ 8.1578e-05$ & $5.7080e-05$& $1.0798e-04$  & $ 8.3704e-05$ &$  5.7304e-05$ & $1.0842e-04$\\
			
			& $1/2924$   & $5.3962e-05$ &$3.6999e-05$ & $6.5334e-05$  &$5.2238e-05 $& $ 3.6101e-05$ & $6.5537e-05$ \\\hline
			\multicolumn{2}{|c|}{\text { order }}
			& $3.0275 $           & $2.9688 $&$3.2830$ &$ 3.1313 $&$ 3.0501$& $3.2733$\\
			
			\hline                & $1/2080$   & $1.5087e-04$ & $1.0163e-04$ & $2.0010e-04$   & $ 1.5183e-04$ & $1.0209e-04 $& $2.0010e-04$\\
			
		$5$  & $1/2520$   &$ 8.1578e-05$ & $5.7080e-05$ & $1.0798e-04$   & $ 8.3704e-05$ &$  5.7304e-05$ & $1.0842e-04$\\
			
			& $1/2924$   & $5.3962e-05$ &$3.6999e-05$ & $6.5334e-05$   &$5.2238e-05 $& $ 3.6101e-05$ & $6.5537e-05$ \\\hline
			\multicolumn{2}{|c|}{\text { order }}
			& $3.0275 $             & $2.9688 $&$3.2830$ &$ 3.1313 $&$ 3.0501$& $3.2733$ \\
			\hline
	\end{tabular}}
\end{table}

\section{Conclusions} \label{section:conclusion}
In this paper, we propose a third-order implicit-explicit Runge-Kutta (IMEX-RK3) numerical method to solve the Landau-Lifshitz equation. By introducing an artificial damping term, IMEX-RK method can achieve higher-order accuracy in time, with the order conditions satisfied. In the framework, we construct the third-order implicit-explicit Runge-Kutta scheme, and the stability condition is imposed. Moreover, in spite of the multi-stage nature and its complicated nonlinear terms, a rigorous optimal rate convergence analysis of this IMEX-RK3 method is provided. It is worth mentioning that the convergence analysis is valid for all damping parameter $\alpha >0$. In addition, its numerical accuracy and the insensitive dependence on the artificial damping parameter $\alpha$ have been verified in both the 1-D and 3-D computations. Numerical results have demonstrated that the IMEX-RK3 method works well for a general damping parameter, regardless of the small damping parameters in real micromagnetics simulations or the large damping parameters in theoretical works. In summary, the proposed numerical scheme not only preserves higher order accuracy and higher computational efficiency, but also its stability is not restricted by the magnitude of damping parameters, in comparison with many existing numerical methods.

\section*{Acknowledgments}
This work is supported in part by the grants NSFC  12271360 and 11501399 (R.~Du) and NSF DMS-2012269, DMS-2309548 (C.~Wang).



\begin{thebibliography}{99}




\bibitem{Lubich2021}
{\sc G. Akrivis, M. Feischl, B. Kov\'acs, and C. Lubich},  {\em Higher--order linearly implicit full discretization of the {L}andau--{L}ifshitz--{G}ilbert equation}, Math. Comput., 90 (2021), pp. 995--1038.

\bibitem{alouges2006convergence}
{\sc F. Alouges and P. Jaisson},  {\em Convergence of a finite element discretization for the {L}andau--{L}ifshitz equations in micromagnetism}, Math. Models Methods Appl. Sci., 16 (2006), pp. 299--316.

\bibitem{an2016optimal}
{\sc R.~An}, {\em Optimal error estimates of linearized {C}rank--{N}icolson {G}alerkin method for {L}andau--{L}ifshitz equation}, J. Sci. Comput., 69 (2016), pp. 1--27.

\bibitem{An2021}
{\sc R. An, H. Gao, and W. Sun}, {\em Optimal error analysis of {E}uler and {C}rank--{N}icolson projection finite difference schemes for {L}andau--{L}ifshitz equation}, SIAM J. Numer. Anal., 59 (2021), pp. 1639--1662.

\bibitem{An2022}
{\sc R. An and W. Sun}, {\em Analysis of backward {E}uler projection {FEM} for the {L}andau--{L}ifshitz equation}, IMA J. Numer. Anal., 42 (2022), pp. 2336--2360.

\bibitem{ascher1997implicit}
{\sc U.M. Ascher, S.J. Ruuth, and R.J. Spiteri}, {\em Implicit--explicit {R}unge--{K}utta methods for time--dependent partial differential equations}, Appl. Numer. Math., 25 (1997), pp. 151--167.


\bibitem{bartels2006convergence}
{\sc S.~Bartels and P.~Andreas}, {\em Convergence of an implicit finite element method for the {L}andau--{L}ifshitz--{G}ilbert equation}, SIAM J. Numer. Anal.,  44 (2006), pp. 1405--1419.

\bibitem{bartels2008numerical}
{\sc S.~Bartels, J.~Ko, and A.~Prohl}, {\em Numerical analysis of an explicit approximation scheme for the {L}andau--{L}ifshitz--{G}ilbert equation}, Math. Comput., 77 (2008), pp. 773--788.

\bibitem{baskaran13b}
{\sc A.~Baskaran, J.~Lowengrub, C.~Wang, and S.~Wise},
{\em Convergence analysis of a second order convex splitting scheme for the modified phase field crystal equation}, SIAM J. Numer. Anal., 51 (2013), PP. 2851--2873.

\bibitem{boscarino2016high}
{\sc S. Boscarino, F. Filbet, and G. Russo}, {\em High order semi--implicit schemes for time dependent partial differential equations}, J. Sci. Comput., 68 (2016), PP. 975--1001.

\bibitem{Brown1963micromagnetics}{\sc W.F. {Brown}}, {\em Micromagnetics}, Springer, New York, NY, 1963.

\bibitem{cai2022second}
{\sc Y. Cai, J. Chen, C. Wang, and C. Xie}, {\em A second--order numerical method for {L}andau--{L}ifshitz--{G}ilbert equation with large damping parameters}, J. Comput. Phys., 451 (2022), PP. 110831.

\bibitem{chen2021convergence}
{\sc J. Chen, C. Wang, and C. Xie}, {\em Convergence analysis of a second--order semi--implicit projection method for {L}andau--{L}ifshitz equation}, Appl. Numer. Math., 168 (2021), PP. 55--74.

\bibitem{chen16}
{\sc W.~Chen, Y.~Liu, C.~Wang, and S.~Wise}, {\em An optimal-rate convergence analysis of a fully discrete finite difference scheme for {Cahn--Hilliard--Hele--Shaw} equation}, Math. Comp., 85 (2016), pp. 2231--2257.

\bibitem{chenW20a}
{\sc W.~Chen, C.~Wang, S.~Wang, X.~Wang, and S.~Wise}, {\em Energy stable numerical schemes for ternary {Cahn--Hilliard} system}, J. Sci. Comput., 84 (2020), pp. 27.

\bibitem{Ciarlet1978}
{\sc P.~Ciarlet}, {\em The Finite Element Method for Elliptic Problems}, Elsevier Science, 1978.

\bibitem{cimrak2005error}
{\sc I. Cimr{\'a}k}, {\em Error estimates for a semi--implicit numerical scheme solving the {L}andau--{L}ifshitz equation with an exchange field}, IMA J. Numer. Anal., 25 (2005), pp. 611--634.

\bibitem{cimrak2007survey}
{\sc I. Cimr{\'a}k}, {\em A survey on the numerics and computations for the {L}andau--{L}ifshitz equation of micromagnetism}, Arch. Comput. Methods Eng., 15 (2008), pp. 277--309.

\bibitem{duan22b}
{\sc C.~Duan, W.~Chen, C.~Liu, C.~Wang, and X.~Yue}, {\em A second order accurate, energy stable numerical scheme for one--dimensional porous medium equation by an energetic variational approach}, Commun. Math. Sci., 20 (2022), pp. 987--1024.

\bibitem{duan22a}
{\sc C.~Duan, W.~Chen, C.~Liu, C.~Wang, and S.~Zhou}, {\em Convergence analysis of structure--preserving numerical methods for nonlinear {F}okker--{P}lanck equations with nonlocal interactions}, Math. Methods Appl. Sci., 45 (2022), pp. 3764--3781.

\bibitem{duan20a}
{\sc C.~Duan, C.~Liu, C.~Wang, and X.~Yue}, {\em Convergence analysis of a numerical scheme for the porous medium equation by an energetic variational approach}, Numer. Math. Theor. Meth. Appl., 13 (2020), pp. 1--18.

\bibitem{E95}
{\sc W.~E and J.~Liu},  {\em Projection method {I}: {Convergence} and numerical boundary layers}, SIAM J. Numer. Anal., 32 (1995), pp. 1017--1057.

\bibitem{E02}
{\sc W.~E and J.~Liu},  {\em Projection method {III}. {Spatial} discretization on the staggered grid}, Math. Comput., 71 (2002), pp. 27--47.

\bibitem{weinan2001numerical}
{\sc W. E and X. Wang}, {\em  Numerical methods for the {L}andau--{L}ifshitz equation}, SIAM J.Numer.Anal., 38 (2000), pp. 1647--1665.

\bibitem{fuwa2012finite}
{\sc A. Fuwa, T. Ishiwata, and M. Tsutsumi},  {\em Finite difference scheme for the {L}andau--{L}ifshitz equation}, Jpn.J.Ind.Appl.Math., 29 (2012), pp. 83--110.

\bibitem{gao2014optimal}
{\sc H. Gao},  {\em Optimal error estimates of a linearized backward {E}uler FEM for the {L}andau--{L}ifshitz equation}, SIAM J. Numer. Anal., 52 (2014), pp. 2574--2593.

\bibitem{CJreview2007}
{\sc Carlos J. Garc\'{i}a-Cervera},  {\em Numerical micromagnetics: a review}, SeMA J., 39 (2007), pp. 103--135.

\bibitem{gilbert1955lagrangian}
{\sc T.~Gilbert}, {\em A {L}agrangian formulation of the gyromagnetic equation of the magnetization field}, Phys. Rev., 100 (1955), pp. 1243.

\bibitem{Girault1986}
{\sc V. Girault and P.A. Raviart}, {\em Finite Element Methods for {N}avier--{S}tokes Equations, Theorems and Algorithms}, Springer-Verlag, 1986.

\bibitem{guan17a}
{\sc Z.Guan, J.S. Lowengrub, and C. Wang},  {\em Convergence analysis for second order accurate schemes for the periodic nonlocal {A}llen--{C}ahn and {C}ahn--{H}illiard equations}, Math.Method.Appl.Sci., 40 (2017), pp. 6836--6863.

\bibitem{guan14a}
{\sc Z.~Guan, C.~Wang, and S.~Wise},  {\em A convergent convex splitting scheme for the periodic nonlocal {Cahn--Hilliard} equation}, Numer. Math., 128 (2014), pp. 377--406.

\bibitem{Gui24a}
{\sc Y. Gui, C. Wang, and J. Chen},  {\em Implicit-explicit {Runge-Kutta} methods for {Landau-Lifshitz} equation with arbitrary damping}, Commun. Math. Sci., 2024, accepted and in press.

\bibitem{2006Solving}
{\sc E. Hairer and G. Wanner}, {\em Solving ordinary differential equations}, Springer Berlin, Heidelberg, 1993.

\bibitem{kruzik2006recent}
{\sc M. Kruz{\'\i}k and A. Prohl},  {\em Recent developments in the modeling, analysis, and numerics of ferromagnetism}, SIAM Rev., 48 (2006), pp. 439--483.

\bibitem{Landau1935On}
{\sc L. Landau and E. Lifshitz}, {\em On the theory of the dispersion of magnetic permeability in ferromagnetic bodies}, Phys. Zeit. der Sowj., 8 (1935), pp. 153--169.


\bibitem{NMTMA-16-182}
{\sc P. Li, L. Yang, L. Jin, R. Du, and J. Chen}, {\em A Second-Order Semi-Implicit Method for the Inertial Landau-Lifshitz-Gilbert Equation}, Numer. Math. Theory Methods Appl., 16 (2023), pp. 182--203.

\bibitem{LiX21a}
{\sc X. Li, Z. Qiao, and C. Wang},  {\em Convergence analysis for a stabilized linear semi-implicit numerical scheme for the nonlocal {Cahn-Hilliard} equation}, Math. Comput., 90 (2021), pp. 171--188.

\bibitem{LiX23a}
{\sc X. Li, Z. Qiao, and C. Wang},  {\em Stabilization parameter analysis of a second order linear numerical scheme for the nonlocal {Cahn-Hilliard} equation}, IMA J. Numer. Anal., 43 (2023), pp. 1089--1114.

\bibitem{LiX24a}
{\sc X. Li, Z. Qiao, and C. Wang},  {\em Double stabilizations and convergence analysis of a second-order linear numerical scheme for the nonlocal {Cahn-Hilliard} equation}, Sci. China Math., 67 (2024), pp. 187--210.

\bibitem{2010Implicit}
{\sc L. Pareschi and G.Russo},  {\em Implicit--explicit {R}unge--{K}utta schemes and applications to hyperbolic systems with relaxation}, J. Sci. Comput., 25 (2005), pp. 129--155.

\bibitem{STWW2003}
{\sc R. Samelson, R. Temam, C. Wang, and S. Wang},  {\em Surface pressure {P}oisson equation formulation of the primitive equations: {Numerical} schemes}, SIAM J. Numer. Anal., 41 (2003), pp. 1163--1194.

\bibitem{STWW2007}
{\sc R. Samelson, R. Temam, C. Wang, and S. Wang},  {\em A fourth order numerical method for the planetary geostrophic equations with inviscid geostrophic balance}, Numer. Math., 107 (2007), pp. 669--705.

\bibitem{temam01}
{\sc R. Temam}, {\em Navier-Stokes Equations: Theory and Numerical Analysis}, American Mathematical Society, 2001.

\bibitem{Wang2000}
{\sc C. Wang and J.G. Liu},  {\em Convergence of gauge method for incompressible flow}, Math. Comput., 69 (2000), pp. 1385--1407.

\bibitem{Wang2002}
{\sc C. Wang and J.G. Liu},  {\em Analysis of finite difference schemes for unsteady {Navier-Stokes} equations in vorticity formulation}, Math. Comput., 91 (2002), pp. 543--576.

\bibitem{Wang2004}
{\sc C. Wang, J.G. Liu, and H. Johnston},  {\em Analysis of a fourth order finite difference method for incompressible {B}oussinesq equation}, Numer. Math., 97 (2004), pp. 555--594.


\bibitem{wang2020local}
{\sc H. Wang, Q. Zhang, S. Wang, and C.W. Shu},  {\em Local discontinuous {G}alerkin methods with explicit--implicit--null time discretizations for solving nonlinear diffusion problems}, Sci. China Math., 63 (2020), pp. 183--204.

\bibitem{WangL15}
{\sc L. Wang, W. Chen, and C. Wang},  {\em An energy-conserving second order numerical scheme for nonlinear hyperbolic equation with an exponential nonlinear term}, J. Comput. Appl. Math., 280 (2015), pp. 347--366.

\bibitem{wang2001gauss}
{\sc X.P. Wang, C.J. Garc{\i}a-Cervera, and W. E},  {\em A {G}auss--{S}eidel projection method for micromagnetics simulations}, J. Comput. Phys., 171 (2001), pp. 357--372.


\bibitem{xie2020second}
{\sc C. Xie, C.J. Garc{\i}a-Cervera, C. Wang, Z. Zhou, and J. Chen},  {\em Second--order semi--implicit projection methods for micromagnetics simulations}, J. Comput. Phys., 404 (2020), pp. 109104.

\bibitem{EAJAM-14-601}
{\sc Z.Y. Fang and X.P. Wang},  {\em An Adaptive Moving Mesh Method for Simulating Finite-Time Blowup Solutions of the Landau-Lifshitz-Gilbert Equation}, East Asian J. Appl. Math., 14 (2024), pp. 601--635.



\end{thebibliography}
\end{document}